\UseRawInputEncoding

\documentclass[reqno]{amsart}
\usepackage{amsfonts}
\usepackage{hyperref}
\usepackage{mathrsfs}
\usepackage{color}
\usepackage{cite}
\usepackage{amscd}
\usepackage{latexsym}
\usepackage{amsfonts}
\usepackage{graphicx}
\usepackage{CJK,indentfirst,amsmath,amsfonts,amssymb,amsthm,cite,cases, subeqnarray,setspace}

\allowdisplaybreaks[4]

\newcommand{\ie}{i.e.}

\newcommand{\et}{\textit{et al.}}
\newcommand{\norm}[1]{\left\lVert#1\right\rVert}
\newcommand{\abs}[1]{\left|#1\right|}
\newcommand{\xkh}[1]{\left(#1\right)}

\newcommand{\dkh}[1]{\left\{#1\right\}}
\newcommand{\lan}[1]{\left\langle#1\right\rangle}

\begin{document}
	\title[Linear stability analysis of the Couette flow for 2D compressible NSP]
		  {Linear stability analysis of the Couette flow for 2D compressible Navier-Stokes-Poisson system}
	\author{Yurui Lu\ \ \ \ and\ \ \ \  Xueke Pu}
	
	\address{Yurui Lu \newline
		School of Mathematics and Information Sciences, Guangzhou University, Guangzhou 510006, P.R. China}
	\email{ luyuruigzhu@163.com}

	\address{Xueke Pu \newline
		School of Mathematics and Information Sciences, Guangzhou University, Guangzhou 510006, P.R. China}
	\email{puxueke@gmail.com}

	\thanks{This work was supported in part by the National Natural Science Foundation of China (12471220) and the Natural Science Foundation of Guangdong Province of China (2024A1515012467).}

\begin{abstract}
	In this paper, we study the linear stability of Couette flow for 2D compressible Navier-Stokes-Poisson system at high Reynolds number in the domain $\mathbb{T}\times\mathbb{R}$ with initial perturbation in Sobolev spaces. We establish the upper bounds for the solutions of linearized system near Couette flow. In particular, we show that the irrotational component of the perturbation may have a transient growth, after which it decays exponentially.
\end{abstract}

\maketitle \numberwithin{equation}{section}
\newtheorem{proposition}{Proposition}[section]
\newtheorem{theorem}{Theorem}[section]
\newtheorem{lemma}[theorem]{Lemma}
\newtheorem{remark}[theorem]{Remark}
\newtheorem{hypothesis}[theorem]{Hypothesis}
\newtheorem{definition}{Definition}[section]
\newtheorem{corollary}{Corollary}[section]
\newtheorem{assumption}{Assumption}[section]
\section{\textbf{Introduction}}
\setcounter{section}{1}\setcounter{equation}{0}
	In this paper, we consider the compressible two-fluid Navier-Stokes-Poisson system 
	\begin{align}
		\label{two-fluid model equ 1}
		&\partial_{t}\tilde{\eta}_{\pm}+\nabla\cdot(\tilde{\eta}_{\pm}\tilde{v}_{\pm})=0,   \\
		&\tilde{\eta}_{\pm}m_{\pm}(\partial_{t}\tilde{v}_{\pm}+\tilde{v}_{\pm}\cdot\nabla\tilde{v}_{\pm})+\frac{1}{M^2}T_{\pm}\nabla\tilde{\eta}_{\pm}=\mp e \tilde{\eta}_{\pm}\nabla\tilde{\phi}+\nu\Delta\tilde{v}_{\pm}+\lambda\nabla \mathrm{div} \tilde{v}_{\pm},   \\
		\label{two-fluid model equ 3}
		&-\Delta\tilde{\phi}=4\pi e(\tilde{\eta}_{+}  -  \tilde{\eta}_{-} ) ,
	\end{align}
	in 2D domain $ \mathbb{T}\times \mathbb{R} $, where $\mathbb{T}=\mathbb{R}/ \mathbb{Z}$. Here, $\tilde{\eta}_{\pm}$, $\tilde{v}_{\pm}$, $m_{\pm}$, $T_{\pm}$ denote the density, velocity, mass, effective temperature, respectly, of ion fluid ($+$) and electron fluid ($-$), and $\tilde{\phi}$ denotes the electrostatic potential. The constant $M$ is the Mach number and $\nu, \lambda$ are shear and bulk viscosity coefficients of the fluid, respectly. Such a system is a fundamental model in plasma physics, derived from Navier-Stokes-Maxwell system as the light speed tends to infinity. It describes the dynamics of ion fluid and electron fluid with interaction of self-consistent electric field.
	
	Since the mass ratio of electron and ion has a scale of $10^{-3} \ll 1$ , the system (\ref{two-fluid model equ 1})-(\ref{two-fluid model equ 3}) can be simplified to the following one-fluid Navier-Stokes-Poisson system under ion background of zero velocity and constant density
	\begin{align}
		\label{elec NSP equ 1}
		&\partial_{t}\tilde{\eta}_{-}+\nabla\cdot(\tilde{\eta}_{-}\tilde{v}_{-})=0,   \\
		&\tilde{\eta}_{-}(\partial_{t}\tilde{v}_{-}+\tilde{v}_{-}\cdot\nabla\tilde{v}_{-})+\frac{1}{M^2}\nabla\tilde{\eta}_{-}= \tilde{\eta}_{-}\nabla\tilde{\phi}+\nu\Delta\tilde{v}_{-}+\lambda\nabla \mathrm{div} \tilde{v}_{-},   \\
		\label{elec NSP equ 3}
		&-\Delta\tilde{\phi}=4\pi(1  -  \tilde{\eta}_{-} ) .
	\end{align}
	As shown in \cite{Guo CMP 2011}, letting electron mass $m_{-} \longrightarrow 0$, we obtain the Boltzmann relation for electron 
	\begin{align}
		\tilde{\eta}_{-}=\eta_{0} e^{\frac{e}{T_{-}}M^2 \tilde{\phi}} . \nonumber
	\end{align}
	Then we deduce the simplified one-fluid ion Navier-Stokes-Poisson system
	\begin{align}
		\label{ion NSP equ 1}
		&\partial_{t}\tilde{\eta}_{+}+\nabla\cdot(\tilde{\eta}_{+}\tilde{v}_{+})=0,   \\
		&\tilde{\eta}_{+}(\partial_{t}\tilde{v}_{+}+\tilde{v}_{+}\cdot\nabla\tilde{v}_{+})+\frac{1}{M^2}\nabla\tilde{\eta}_{+}=-  \tilde{\eta}_{+}\nabla\tilde{\phi}+\nu\Delta\tilde{v}_{+}+\lambda\nabla \mathrm{div} \tilde{v}_{+},   \\
		\label{ion NSP equ 3}
		&-\Delta\tilde{\phi}=4\pi (\tilde{\eta}_{+}  -  e^{M^2 \tilde{\phi}} ) .
	\end{align}
	It's clear that 
	\begin{align}
		(\tilde{\eta}_s,\tilde{v}_s,\tilde{\phi}_s)=(1,(y,0)^T,0) \nonumber
	\end{align}
	is a stationary solution for (\ref{elec NSP equ 1})-(\ref{elec NSP equ 3}) and (\ref{ion NSP equ 1})-(\ref{ion NSP equ 3}) and this solution  is called Couette flow. 
	 
	It is natural and important to study the stability properties of Couette flows for Navier-Stokes-Poisson system in \textit{hydrodynamic stability theory}\cite{hydrodynamic stability}. Justification stability analysis to plasma equilibrium states for Vlasov-Poisson and Vlasov-Maxwell systems in one, two and three dimensions was already studied by Holm \et \cite{plasma equilibria}. To the best of our knowledge, there is no literature about the rigorous stability analysis of Couette flow for 2D compressible Navier-Stokes-Poisson system until now and this is a fundamental and interesting problem in plasma fluid dynamics. Thus, in our paper, we are concerned with the linear stability of Couette flow for 2D compressible Navier-Stokes-Poisson system, (\ref{elec NSP equ 1})-(\ref{elec NSP equ 3}) and (\ref{ion NSP equ 1})-(\ref{ion NSP equ 3}), in the domain $\mathbb{T}\times \mathbb{R}$.

	To begin, let us introduce some related mathematical results on the stability analysis of Couette flow in fluid dynamics. For an incompressible fluid, Rayleigh\cite{Rayleigh} and Kelvin\cite{Kelvin 1887} obtained the classical results on the linear stability analysis of Couette flow. Bedrossian and Masmoudi\cite{Bedrossian_Masmoudi nonlinear} rigorously confirmed the inviscid damping and  the asymptotic stability of planar Couette flow  for Euler equations in the domain $\mathbb{T} \times \mathbb{R}$ at the nonlinear level. This result has motivated a number of researches on the inviscid damping and stability analysis for linearized systems around other stationary solutions under different physical settings. For example, Wei \et \cite{Wei zhang zhao} gave the decay estimate of the velocity field and $H^1$ scattering for 2D linearized Euler system around a class of monotone shear flows in the finite channel $\mathbb{T} \times [0,1]$ and C. Zillinger\cite{Zillinger 2017} proved the stability, scattering and inviscid damping for 2D Euler systems around a large class of strictly monotone shear flows in periodic channel in Sobolev space. One can also refer to other nice results \cite{Grenier JFA}\cite{Hao Jia inviscid damping in Gevrey} \cite{Wei Zhang Zhao 2019 APDE} \cite{Zillinger 2016}. When considering viscous problem, some interesting results have been obtained in the study of transition threshold and enhanced dissipation, namely faster decay rate than heat equation. The great progress  was made by Bedrossian \et\cite{Bedrossian threshold 3D}\cite{Bedrossian AMS I}\cite{Bedrossian AMS II}, where they studied the stability threshold for 3D incompressible Navier-Stokes equations on $\mathbb{T} \times \mathbb{R} \times \mathbb{T}$ in Sobolev regularity. For $\mathbb{T} \times \mathbb{R}$, we can refer to \cite{Bedrossian 2D}. Later, Chen \et\cite{Chen Wei Zhang} studied this problem in $\mathbb{T} \times (-1,1)$ at high Reynolds number. The method they developed to establish resolvent estimates and space-time estimates can explain enhanced dissipation effect, inviscid damping effect and boundary layer effect. When considering nonlinear stability, due to the nonlinear interaction and resonance, which may exist and lead to the instability in nonlinear system, the analysis becomes challenging. For a general class of monotone shear flows, in the case of not closing to Couette flow, Ionescu A D and Jia\cite{Ionescu and Jia H} established the inviscid damping and nonlinear asymptotic stability for Euler system in a finite channel $\mathbb{T}\times [0,1]$.
	Masmoudi and Zhao\cite{Masmoudi and Zhao 2024 Annals} got the similar results from a different method. 	The transition threshold and nonlinear stability of Couette flow for 3D Navier-Stokes on $\mathbb{T} \times \mathbb{R} \times \mathbb{T}$ were studied by Wei and Zhang\cite{Wei_Zhang_3D_nonlinear}. For domain $\mathbb{T}\times [-1,1]\times\mathbb{T}$ at nonlinear level, one can refer to \cite{Chen Wei Zhang nonlinear channel}.
	
	The stability analysis of Couette flow for compressible fluid has been investigated since the 1940s\cite{Lin C C} in physics literature for nearly a century\cite{plasma equilibria}\cite{hydrodynamic stability}. 
	Glatzel \cite{Glatzel 1988} has studied the inviscid linear stability property of a simple  shear layer with supersonic velocity difference across the layer using normal mode method and also investigated the viscous linear stability of  Couette flow in the same method\cite{Glatzel}. 
	With temperature considered, Duck \et \cite{Duck} studied the linear stability property of plane Couette flow for realistic fluid models, where they did extensive analysis in the inviscid regime and showed that the details of the mean flow profile have a profound effect on the stability of the flow. 
	The linear stability of viscous Couette flow with a perfect compressible gas model was studied by Hu and Zhong \cite{Hu & Zhong} using high-order finite-difference global method and Chebyshev spectral collocation global method. 
	For rigorous mathematical analysis, we refer to the following literature.
	Kagei\cite{Kagei2011} showed that the plane Couette flow of compressible Navier-Stokes equations in infinite layer $\Omega=\mathbb{R}^{n-1} \times (0,l)$ is asymptotically stable, provided that the initial disturbances, the Reynolds number and Mach number are sufficiently small. 
	Under the same assumption, the global existence of strong solution to the compressible Navier-Stokes equations around parallel flows in $\mathbb{R}^{n-1} \times (0,l)$ was also proved by Kagei \cite{Kagei JDE 2011}.
	Moreover, by spectral analysis and energy method, the nonlinear stability property was proved in case of $n=2$ and the large time behavior is described by 1D viscous Burgers equation \cite{Kagei2012ARMA}.
	As the Navier-slip boundary is involved, the constraint on Reynolds and Mach numbers required in \cite{Kagei2011} to guarantee the stability of the plane Couette flow can be relaxed, proved by Li and Zhang\cite{Li and Zhang JDE}.
	The inviscid damping and enhanced dissipation phenomena were studied by Antonelli \et\cite{Antonelli 2021}\cite{Antonelli arxiv} for homogeneous Couette flow of 2D isentropic compressible Navier-Stokes equations in the domain $\mathbb{T} \times \mathbb{R}$. 
	Later on, for 3D isentropic compressible Navier-Stokes equations on $\mathbb{T} \times \mathbb{R} \times \mathbb{T}$, Zeng \et\cite{Zeng and Zhang and Zi} proved the enhanced dissipation phenomenon and lift-up phenomenon around Couette flow at linear level.
	For non-isentropic case, Zhai\cite{Zhai} showed the inviscid damping phenomenon of compressible Euler system on $\mathbb{T} \times \mathbb{R}$.
	In addition, Pu \et\cite{Pu and Zhou and Bian} studied the linear stability property for Couette flow of the Euler-Poisson system for both ion fluid and electron fluid on $\mathbb{T} \times \mathbb{R}$.
	After we initiated this work, we noticed the work of Huang \et\cite{Huang Feimin nonlinear Couette}, who studied the nonlinear stability of Couette flow  for 2D compressible Navier-Stokes equations in Sobolev space at high Reynolds number in $\mathbb{T} \times \mathbb{R}$ and showed the enhanced dissipation phenomenon and stability threshold for compressible Couette flow.
	
	Despite the long-standing research on the stability analysis of compressible flow, with considerable work established, the rigorous mathematical result on the compressible Couette flow is less developed compared to the incompressible case. Specifically, the linear stability analysis of Couette flow for compressible Navier-Stokes-Poisson system is unclear until now, which will be investigated in this paper.

	\subsection{Ion case} 
	\ 
	
	In order to study the linear stability properties of (\ref{ion NSP equ 1})-(\ref{ion NSP equ 3}), we consider a perturbation around Couette flow, i.e., let
	\begin{align}
		\tilde{\eta}_{+}=\eta_{+}+\tilde{\eta}_s, \quad  \tilde{v}_{+}=u_{+}+\tilde{v}_s , \quad \tilde{\phi}=\phi+\tilde{\phi}_s. \nonumber
	\end{align}
	The linearized system around the Couette flow reads as follows
	\begin{align}
		\label{linearized equ 1}
		&\partial_{t}\eta_{+} +y\partial_x\eta_{+}+\nabla\cdot u_{+}=0,\\
		&\partial_{t}u_{+}+y\partial_xu_{+}+
		\begin{pmatrix}
			u_{+}^y\\
			0
		\end{pmatrix}
		+\frac{1}{M^2}\nabla\eta_{+}=-\nabla\phi+\nu\Delta u_{+} +\lambda\nabla(\nabla\cdot u_{+}),\\
		\label{linearized equ 3}
		&\Delta \phi=4\pi(M^2 \phi-\eta_{+}).
	\end{align}
	
	Let $\psi_{+}=\nabla\cdot u_{+}$, $\omega_{+}=\nabla^\perp\cdot u_{+}$, where $\nabla^\perp=(-\partial_y,\partial_x)^T$. We use Helmholtz projection $\mathbb{P}$ to decompose the velocity field $u_{+}$ :
	$$u_{+}=(u^x_{+},u^y_{+})^T=\nabla\Delta^{-1}\psi_{+}+\nabla^\perp\Delta^{-1}\omega_{+}=:\mathbb{Q}[u_{+}]+\mathbb{P}[u_{+}].$$
	Then the linearized system \eqref{linearized equ 1}-\eqref{linearized equ 3} is transformed into the form with respect to $(\eta_{+},\psi_{+},\omega_{+})$
	\begin{align}
		\label{linear sys eta psi om 1}
		&\partial_{t}\eta_{+}+y\partial_x\eta_{+}+\psi_{+}=0,\\
		\label{linear sys eta psi om 2}
		&\partial_{t}\psi_{+}+y\partial_x\psi_{+}+2\partial_xu^y_{+}+\frac{1}{M^2}\Delta\eta_{+}=-4\pi \Delta (-\Delta +4 \pi M^2)^{-1} \eta_{+}+(\nu+\lambda)\Delta \psi_{+},\\
		\label{linear sys eta psi om 3}
		&\partial_{t}\omega_{+}+y\partial_x\omega_{+}-\psi_{+}=\nu\Delta\omega_{+}.
	\end{align}

	\subsection{Electron case}
	\  
	
	As shown above, we consider a perturbation around Couette flow for (\ref{elec NSP equ 1})-(\ref{elec NSP equ 3}). Let
	\begin{align}
		\tilde{\eta}_{-}=\eta_{-}+\tilde{\eta}_s, \quad  \tilde{v}_{-}=u_{-}+\tilde{v}_s , \quad \tilde{\phi}=\phi+\tilde{\phi}_s. \nonumber
	\end{align}
	Then the linearized system around Couette flow for  (\ref{elec NSP equ 1})-(\ref{elec NSP equ 3}) can be written as
	\begin{align}
		\label{linearized elec equ 1}
		&\partial_{t}\eta_{-} +y\partial_x\eta_{-}+\nabla\cdot u_{-}=0,\\
		&\partial_{t}u_{-}+y\partial_xu_{-}+
		\begin{pmatrix}
			u_{-}^y\\
			0
		\end{pmatrix}
		+\frac{1}{M^2}\nabla\eta_{-}=\nabla\phi+\nu\Delta u_{-} +\lambda\nabla(\nabla\cdot u_{-}),\\
		\label{linearized elec equ 3}
		&\Delta \phi=4\pi \eta_{-}.
	\end{align}
	
	Let $\psi_{-}=\nabla\cdot u_{-}$, $\omega_{-}=\nabla^\perp\cdot u_{-}$. We also use Helmholtz projection $\mathbb{P}$ to decompose the velocity field $u_{-}$ :
	$$u_{-}=(u^x_{-},u^y_{-})^T=\nabla\Delta^{-1}\psi_{-}+\nabla^\perp\Delta^{-1}\omega_{-}=:\mathbb{Q}[u_{-}]+\mathbb{P}[u_{-}].$$
	Then the linearized system (\ref{linearized elec equ 1})-(\ref{linearized elec equ 3}) can be written in the form with respect to $(\eta_{-},\psi_{-},\omega_{-})$ 
	\begin{align}
		\label{linear elec sys eta psi om 1}
		&\partial_{t}\eta_{-}+y\partial_x\eta_{-}+\psi_{-}=0,\\
		\label{linear elec sys eta psi om 2}
		&\partial_{t}\psi_{-}+y\partial_x\psi_{-}+2\partial_x u^y_{-}+\frac{1}{M^2}\Delta\eta_{-}=4\pi \eta_{-}+(\nu+\lambda)\Delta \psi_{-},\\
		\label{linear elec sys eta psi om 3}
		&\partial_{t}\omega_{-}+y\partial_x\omega_{-}-\psi_{-}=\nu\Delta\omega_{-}.
	\end{align}
	
	\subsection{Main results}\ 
	
	Throughout this paper, we denote the average of a function in the $x$ direction by
	\begin{align}
		f_0(y)=\frac{1}{2\pi}\int_{\mathbb{T}}f(x,y)dx, \nonumber
	\end{align}
	and this is also the zero mode in $x$ direction of function $f$. 
	When using symbol $\lesssim \:$, we are neglecting constants which do not depend on $\nu$, but may depend on $(1+M)^\beta$ or $\mathrm{exp}(M^\beta)$, for some $\beta\geq 1$. However, we keep the constants, which go to zero as $M\to +\infty$, like $\frac{1}{M}$.
	We simply write $C_{in}^{\pm}=C_{in}^{\pm}(\eta^{in}_{\pm},\psi^{in}_\pm,\omega^{in}_{\pm})$ to indicate a suitable combination of Sobolev norms of the initial data. In particular, we use $C^{\pm}_{in,s}$ to denote the following combination of Sobolev norms of the initial data, i.e.
	\begin{align}
		C^{\pm}_{in,s}=\frac{1}{M}\norm{\Pi_{\pm}^{in}}_{H^{s+1}}+\norm{\Psi_{\pm}^{in}}_{H^s}+\norm{F_{\pm}^{in}-\nu M^2 \Psi_{\pm}^{in}}_{H^s}.
	\end{align}
	The definition of $\Pi_{\pm}, \Psi_{\pm}, F_{\pm}$ will be clarified later.
	
	The main results for the ion($+$) and electron($-$) NSP system are as follows.
	\begin{theorem}\label{th1}
		Suppose $\nu, \lambda\geq 0$ and $M>0$ satisfy $\nu+\lambda\leq1/2$, $M\leq \mathrm{min}\{(\nu+\lambda)^{-1},\lambda^{-\frac{1}{2}},\nu^{-\frac{1}{3}}\}$
		 and $M^2\leq \mathrm{min} \{ \frac{1}{\nu+\lambda}, \frac{\nu+\lambda}{16\nu }, \xkh{\frac{\nu+\lambda}{32 \pi \nu } }^\frac{1}{2} \}$. Assume that $\eta_{\pm}^{in}\in H^6(\mathbb{T}\times\mathbb{R})$，and $\psi_{\pm}^{in}, \omega_{\pm}^{in}\in H^5(\mathbb{T}\times\mathbb{R})$. Then 
		\begin{align}
			\label{conclusion1 in th 1}
			&\norm{\xkh{\mathbb{Q}[u_{\pm}]-\mathbb{Q}[u_{\pm}]_0}(t)}_{L^2}+\frac{1}{M}\norm{\eta_{\pm}}_{L^2}\lesssim \lan{t}^\frac{1}{2}\mathrm{e}^{-\frac{1}{32}\nu^\frac{1}{3}t}C_{in}^{\pm} ,\\
			\label{conclusion2 in th 1}
			&\norm{\xkh{\mathbb{P}[u_{\pm}]^x-\mathbb{P}[u_{\pm}]^x_0}(t)}_{L^2}\lesssim M \frac{\mathrm{e}^{-\frac{1}{64}\nu^\frac{1}{3}t}}{\lan{t}^\frac{1}{2}}C_{in}^{\pm}+\frac{\mathrm{e}^{-\frac{1}{12}\nu^\frac{1}{3}t}}{\lan{t}} \norm{\omega^{in}_{\pm}+\eta^{in}_{\pm}}_{H^1} ,\\
			\label{conclusion 3 in Th1}
			&\norm{\mathbb{P}[u_{\pm}]^y(t)}_{L^2}\lesssim M \frac{\mathrm{e}^{-\frac{1}{64}\nu^\frac{1}{3}t}}{\lan{t}^\frac{3}{2}}C_{in}^{\pm}+\frac{\mathrm{e}^{-\frac{1}{12}\nu^\frac{1}{3}t}}{\lan{t}^2} \norm{\omega^{in}_{\pm}+\eta^{in}_{\pm}}_{H^2} .
		\end{align}
	\end{theorem} 
	\begin{remark}
		From the definition of $\mathbb{P}$, we get $\mathbb{P}[u_\pm]^y_0(t)=0$. It follows from (\ref{conclusion1 in th 1}) that $\eta$ and the irrotational component of the nonzero modes of velocity field may have a transient growth of $\nu^{-\frac{1}{6}}$ on a time scale of $\nu^{-\frac{1}{3}}$, after which they decay exponentially. From (\ref{conclusion2 in th 1}) and (\ref{conclusion 3 in Th1}), we observed that for the solenoidal component of the nonzero modes of velocity field, the decay of $y$ direction is better than $x$ direction. One should observe that the results in Theorem \ref{th1} are not closed, with loss of derivatives. The condition $M^2\leq \mathrm{min} \{ \frac{1}{\nu+\lambda}, \frac{\nu+\lambda}{16\nu }, \xkh{\frac{\nu+\lambda}{32 \pi \nu } }^\frac{1}{2} \}$ is needed, since the analysis of $k\neq 0$ modes is based on Theorem \ref{theorem k=0 ion} and Theorem \ref{theorem electron k=0}, see Remark \ref{remark of ion 0 mode initial value} and Remark \ref{remark of electron 0 mode initial value}.
	\end{remark}
	\begin{theorem} \label{th 2}
		Suppose $\nu> 0$, $\lambda\geq 0$ and $M>0$ satisfy $\nu+\lambda\leq \frac{1}{2}$ , $M\leq \mathrm{min}\{(\nu+\lambda)^{-1},\lambda^{-\frac{1}{2}},\nu^{-\frac{1}{3}} \}$
		and  $M^2\leq \mathrm{min} \{ \frac{1}{\nu+\lambda}, \frac{\nu+\lambda}{16\nu }, \xkh{\frac{\nu+\lambda}{32 \pi \nu } }^\frac{1}{2} \}$. Assume that $\eta^{in}_{\pm}\in H^1(\mathbb{T}\times\mathbb{R})$, and $\psi^{in}_{\pm},\omega^{in}_{\pm}\in L^2(\mathbb{T}\times\mathbb{R})$. Then
		\begin{align}
			&\norm{\xkh{\psi_{\pm}-\psi_{\pm,0}}(t)}_{L^2}+\frac{1}{M}\norm{\nabla\xkh{\eta_{\pm}-\eta_{\pm,0}}(t)}_{L^2}+\norm{\xkh{\omega_{\pm}-\omega_{\pm,0}}(t)}_{L^2} \nonumber \\
			&\lesssim\nu^{-\frac{1}{2}}\mathrm{e}^{-\frac{\nu^\frac{1}{3}}{32}t} \xkh{\norm{\xkh{\psi^{in}_{\pm}-\psi^{in}_{\pm,0}}}_{L^2}+\frac{1}{M}\norm{\nabla\xkh{\eta^{in}_{\pm}-\eta^{in}_{\pm,0}}}_{L^2}+\norm{\xkh{\omega^{in}_{\pm}-\omega^{in}_{\pm,0}}}_{L^2}}.
		\end{align}
		In addition, the following inequality holds
		\begin{align}
			&\norm{\xkh{u_{\pm}-u_{\pm,0}}(t)}_{L^2}+\frac{1}{M}\norm{\xkh{\eta_{\pm}-\eta_{\pm,0}}(t)}_{L^2}\nonumber \\
			&\lesssim \nu^{-\frac{1}{6}} \mathrm{e}^{-\frac{\nu^\frac{1}{3}}{32}t} \xkh{\norm{\xkh{\psi^{in}_{\pm}-\psi^{in}_{\pm,0}}}_{L^2}+\frac{1}{M}\norm{\nabla\xkh{\eta^{in}_{\pm}-\eta^{in}_{\pm,0}}}_{L^2}+\norm{\xkh{\omega^{in}_{\pm}-\omega^{in}_{\pm,0}}}_{L^2}}.
		\end{align}
	\end{theorem}
	\begin{remark}
		Compared to results in Theorem \ref{th1}, the decay estimates in Theorem \ref{th 2} are worse, but there is no loss of derivatives in Theorem \ref{th 2}. In a sence, we are trading off time decay for regularity.  
	\end{remark}

	\: \\
	\noindent \textbf{Main difficulties and Strategies}

	There are three main difficulties in our research. Firstly, with the presence of viscosity, it follows from summing up (\ref{linear sys eta psi om 1}) and (\ref{linear sys eta psi om 3}) or summing up (\ref{linear elec sys eta psi om 1}) and (\ref{linear elec sys eta psi om 3}), that the conservation of $\eta+\omega$ is no longer valid for both ion fluid and electron fluid, which played an significant role in the stability analysis of Couette flow for Euler-Poisson equations, by reducing the problem to the analysis of a 2D dynamical system, see\cite{Pu and Zhou and Bian}. Secondly, there are dissipation terms for $\psi$ and $\omega$, 
	see (\ref{linear sys eta psi om 2})(\ref{linear sys eta psi om 3})(\ref{linear elec sys eta psi om 2})(\ref{linear elec sys eta psi om 3}). 
	However, there is no dissipation term in (\ref{linear sys eta psi om 1}) or (\ref{linear elec sys eta psi om 1}), resulting that the priori decay estimate for $\eta$ is nontrivial.  
	 The third difficulty arises from the existence of the self-consistent electric field, which is determined by the density and an elliptic equation and makes the system nonlocal, in contrast to Euler or Navier-Stokes system. 
	 
	 Inspired by the work\cite{Antonelli 2021}, we will replace the non-conserved quantity $\eta+\omega$ by quantity $\eta+\omega-\nu M^2 \psi$, which satisfies a more complicated equation with a good dissipation structure. If we could establish decay estimates on ($\eta$, $\psi$, $\eta+\omega-\nu M^2 \psi$), then by triangle inequality, the decay estimates on ($\eta$, $\psi$, $\omega$) would be obtained. In the subsequent analysis, we decouple the zero mode in $x$ direction from the non-zero modes. 
	 
	 For the zero mode in $x$ direction, we can define suitable coercive energy functionals in terms of ($\eta$, $\psi$, $\eta+\omega-\nu M^2 \psi$), and the decay estimates follow from the Gr\"onwall's type inequalities on such energy functionals. The functionals in our paper are more complicated than \cite{Antonelli 2021}'s, as in our paper the energy functionals must account for the effect of the self-consistent electric field, see Remark \ref{remark on k=0 ion energy} and Remark \ref{remark on k=0 electron energy}.
	 
	 For the other modes, as is standard, we firstly make following change of coordinates to remove the transport terms
	 \begin{align}
	 	X=x-yt, \quad Y=y . \nonumber
	 \end{align}
 	We proceed to take Fourier transformation and shift the analysis from physical space $(x,y)$ to Fourier space $(k,\xi)$. Then, we define a coercive weighted energy functional in terms of ($\eta$, $\psi$, $\eta+\omega-\nu M^2 \psi$), where the weights are time-dependent Fourier multipliers. Firstly, the weights for $\eta$ and $\psi$ are  determined by the symmetrization procedure of inviscid system\cite{Pu and Zhou and Bian}, the weight for $\eta+\omega-\nu M^2 \psi$ is determined by the structure of its equation, and we establish the decay estimates with loss of derivatives using Gr\"onwall's inequality, see Theorem \ref{th1}. Secondly, a thorough analysis reveals that the loss of derivatives comes from the multipliers. We next improve the multipliers. Eventually, we obtain the decay estimates without loss of derivatives utilizing Gr\"onwall's inequality, see Theorem \ref{th 2}. This method is inspired by the work \cite{Antonelli 2021}. However, the weighted energy functional in \cite{Antonelli 2021} does not work for NSP system, due to the effect of the self-consistent electric field.
 	 Precisely, the time derivative of such energy functional generates a term with the coefficient  $M\alpha^\frac{1}{2}\dfrac{4\pi}{\alpha+4\pi M^2 \delta} $. This is a bad term for the following two reasons. Firstly, the lack of viscosity in such term results in the loss of smallness, preventing the absorption from some dissipative terms.
 	 Secondly, the coefficient of such term is not integrable in time, preventing the use of Gr\"onwall's inequality.
	Then we introduce another term to modify the weighted energy functional, whose time derivative generates a term, being opposite sign to the above bad term. In addition, the modified weighted energy functional retains the coercive property.

	\: \\
	\noindent \textbf{Structure of the paper} 
	
	 It is clear that there is only one term distinguishing the structure of the ion fluid model (\ref{linear sys eta psi om 1})-(\ref{linear sys eta psi om 3}) and the structure of the electron fluid model (\ref{linear elec sys eta psi om 1})-(\ref{linear elec sys eta psi om 3}).
	 
	 In section $2$, we will study the dynamics of $k=0$ mode for ion fluid and electron fluid separately. The first reason is that the difference between energy functional for ion case and electron case is 
	 considerable and we will estimate some term appearing in their time derivatives by different method. The second reason is that if we analyze them in a uniform energy functional, 
	 then the terms like $4 \pi \norm{\partial_y^{l} (-\partial_{yy}+4\pi M^2 \delta)^{-\frac{1}{2}}\eta_{\cdot,0}}_{L^2}^2$ would appear in the energy functional,
	 see (\ref{definition of mathcal E^l}) and (\ref{definition of mathcal F^l }), with $\delta=1$ corresponding to ion system and $\delta=0$ corresponding to electron system. When considering $l=0$ for electron fluid, such term will become $4 \pi \norm{(-\partial_{yy})^{-\frac{1}{2}}\eta_{-,0}}_{L^2}^2$. But $\xi$ may be $0$ when considering $k=0$ mode. 
	 
	  In section $3$, we will study the dynamics of $k \neq 0$ modes uniformly for ion fluid and electron fluid, for the reason that their corresponding weighted energy functionals are similar and we bound their time derivatives by a similar method.  
	  In addition, terms like $\alpha^{-\frac{1}{2}} f$ also will appear, and $\alpha \neq 0$ in dynamics of $k \neq 0$ modes, where 
	  	\begin{align}
	  		\alpha(t,k,\xi)=k^2+(\xi-kt)^2, \quad k\in\mathbb{Z}, \: \xi \in \mathbb{R},
	  	\end{align}
	  is the symbol associated to operator $-\Delta_L$, with $\Delta_L:=\partial_{XX}+(\partial_{Y}-t \partial_{X})^2 =\partial_x^2+\partial_y^2 = \Delta $. In subsection $3.1$, we will prove Theorem \ref{th1}, and there is loss of derivatives in the dissipation estimates. In subsection $3.2$, we will prove Theorem \ref{th 2}, in which there is no loss of derivatives.  
	
	\: \\
	\noindent \textbf{Notations}

	The Fourier transform of $f(x,y)$ is given by
	\begin{align}
		\widehat{f}(k,\xi)=\iint_{\mathbb{T} \times \mathbb{R}} f(x,y) \mathrm{e}^{-i(kx+\xi y)} dxdy , \nonumber
	\end{align}
	then
	\begin{align}
		f(x,y)= \sum_{k\in \mathbb{Z}} \int_{\mathbb{R}} \widehat{f}(k,\xi) \mathrm{e}^{i(kx+\xi y)} d\xi .
	\end{align}
	The anisotropic Sobolev space $H^{s_1}_x H^{s_2}_y (\mathbb{T} \times \mathbb{R})$ consists of functions that satisfy
	\begin{align}
		\norm{f}_{H^{s_1}_x H^{s_2}_y} := \xkh{\sum_{k} \int_{\mathbb{R}} \lan{k}^{2s_1} \lan{\xi}^{2s_2} \abs{\widehat{f}(k,\xi)}^2 d\xi  }^{\frac{1}{2}}< + \infty , \nonumber
	\end{align}
	and the norm in Sobolev space $H^s(\mathbb{T} \times \mathbb{R})$ is defined as 
	\begin{align}
		\norm{f}_{H^s}:= \xkh{ \sum_{k} \int_{\mathbb{R}} \lan{k,\xi}^{2s} \abs{\widehat{f}(k,\xi)}^2 d\xi   }^{\frac{1}{2}}, \nonumber
	\end{align}
	where $\lan{a}=(1+a^2)^{\frac{1}{2}}$, $\lan{a,b}=(1+a^2+b^2)^{\frac{1}{2}}$, $\forall a, b \in \mathbb{R}$.
	
	In the new reference frame, we also define the functions
	\begin{align}
		&\Pi_{\pm}(t,X,Y)=\eta_{\pm}(t,X+tY,Y) \nonumber \\
		&\Psi_{\pm}(t,X,Y)=\psi_{\pm}(t,X+tY,Y) \nonumber \\
		&\Gamma_{\pm}(t,X,Y)=\omega_{\pm}(t,X+tY,Y) \nonumber \\
		&F_{\pm}(t,X,Y)=	\Pi_{\pm}(t,X,Y)+\Gamma_{\pm}(t,X,Y) \nonumber .
	\end{align}

\section{\textbf{Analysis of $k=0$ mode}}
		In this section, we consider dynamics of $k=0$ mode NSP. By virtue of the linearity of (\ref{linear sys eta psi om 1})-(\ref{linear sys eta psi om 3}) and (\ref{linear elec sys eta psi om 1})-(\ref{linear elec sys eta psi om 3}), the system of zero mode in $x$ is independent to other modes. As a result, we can decouple the $k=0$ mode from the rest modes in the subsequent discussion. 
	\subsection{$k=0$ mode of Ion system}
	\ 
	
	Projecting (\ref{linear sys eta psi om 1})-(\ref{linear sys eta psi om 3}) onto $k=0$ frequency, we obtain
	\begin{align}
		\label{k=0, ion,equ 1}
		&\partial_{t}\eta_{+,0}+\psi_{+,0}=0,\\
		\label{k=0, ion,equ 2}
		&\partial_{t}\psi_{+,0}+\frac{1}{M^2} \partial_{yy}\eta_{+,0}=-4\pi \partial_{yy} (-\partial_{yy} +4 \pi M^2)^{-1} \eta_{+,0}+(\nu+\lambda)\partial_{yy} \psi_{+,0},\\
		\label{k=0, ion,equ 3}
		&\partial_{t}\omega_{+,0}-\psi_{+,0}=\nu \partial_{yy} \omega_{+,0}.
	\end{align}
	When $\nu, \lambda =0$, plugging (\ref{k=0, ion,equ 1}) into (\ref{k=0, ion,equ 2}), we get $\partial_{t}(\eta_{+,0}+\omega_{+,0})=0$, which implies a conservation on $\eta_{+,0}+\omega_{+,0}$. Applying standard $L^2$ energy estimate, it clearly follows that $(\eta_{+,0}, \psi_{+,0}, \omega_{+,0})=(0,0,0)$, if $(\eta_{+,0}^{in}, \psi_{+,0}^{in}, \omega_{+,0}^{in})=(0,0,0)$.
	For $\nu +\lambda>0$, as shown in \cite{Antonelli 2021}, it is convenient to replace (\ref{k=0, ion,equ 3}) by
	\begin{align}
		\partial_{t}(\eta_{+,0}+\omega_{+,0}-\nu M^2 \psi_{+,0})
		=&\nu \partial_{yy}(\eta_{+,0}+\omega_{+,0}-\nu M^2 \psi_{+,0})\nonumber\\
		&+4\pi \nu M^2\partial_{yy}(-\partial_{yy}+4\pi M^2)^{-1}\eta_{+,0} 
		-\nu \lambda M^2 \partial_{yy} \psi_{+,0}, \nonumber
	\end{align}
	which has a better dissipation structure and we have the following decay estimate.
	\begin{theorem} \label{theorem k=0 ion}
		Suppose $\nu, \lambda \geq 0 $, and $\eta_{+,0}^{in}, \psi_{+,0}^{in}, \omega_{+,0}^{in}$ be the initial data of (\ref{k=0, ion,equ 1})-(\ref{k=0, ion,equ 3}). For any $ l \geq 0$, define
		\begin{align}
			\label{definition of mathcal E^l}
			\mathcal{E}^l (t)=&\norm{\partial_y^l \psi_{+,0}(t)}^2_{L^2}
			+ \norm{\partial_y^{l-1} \psi_{+,0}(t)}^2_{L^2} + \frac{1}{M^2}\xkh{ \norm{ \partial_y^{l+1} \eta_{+,0}(t) }^2_{L^2} + \norm{ \partial_y^l \eta_{+,0}(t)  }^2_{L^2} } \nonumber \\
			& +4\pi \xkh{ \norm{\partial_y^{l+1} (-\partial_{yy}+4\pi M^2)^{-\frac{1}{2}}\eta_{+,0}}_{L^2}^2+ \norm{\partial_y^{l} (-\partial_{yy}+4\pi M^2)^{-\frac{1}{2}}\eta_{+,0}}_{L^2}^2 }\nonumber \\
			& + \norm{\partial_y^l \xkh{\omega_{+,0} +\eta_{+,0} -\nu M^2 \psi_{+,0} }(t)}^2_{L^2} . 
		\end{align}
		If $0<\nu+\lambda \leq 1, M^2(\nu +\lambda) \leq 1 , 0\leq \frac{16 \nu M^2}{\nu+\lambda}\leq 1  $, then the following decay estimate holds
		\begin{align}
			\mathcal{E}^l (t) \leq \dfrac{\widetilde{C} \mathcal{E}^l_{in}}{\xkh{\nu C_l t+1 }^{l}} ,
		\end{align}
		where $C_l$ is a constant depending on $l$, with $C_0=0$, and $\widetilde{C}$ is a constant independent to $l$. 
	\end{theorem}
	
	\begin{remark}
		\label{remark on k=0 ion energy}
		When $l=0$, $\partial_y ^{l-1} \psi_{\pm,0}= u_{\pm, 0} ^{y}$ .
		It is natural to consider
		\begin{align}
			\label{natural energy}
			\mathcal{E}(t)=\norm{\partial_y^l \psi_{+,0}(t)}^2_{L^2}
			+\frac{1}{M^2} \norm{ \partial_y^{l+1} \eta_{+,0}(t) }^2_{L^2}
			+4\pi \norm{\partial_y^{l+1} (-\partial_{yy}+4\pi M^2)^{-\frac{1}{2}}\eta_{+,0}}_{L^2}^2,
		\end{align}
		when applying standard energy estimate on (\ref{k=0, ion,equ 2}). Directly computing time derivative, we get 
		\begin{align}
			\label{derivative of natural energy}
			\frac{\mathrm{d}}{\mathrm{dt}}\mathcal{E}(t)= -(\nu+\lambda) \norm{\partial_y^{l+1}\psi_{+,0}}_{L^2}.  
		\end{align}
		In order to apply Gr\"onwall's inequality, we need to modify $\mathcal{E}(t)$ as follow
		\begin{align}
			E_1(t)=\mathcal{E}(t)- \frac{\nu+\lambda}{2}\lan{\partial_y^l \eta_{+,0}(t) , \partial_y^l \psi_{+,0}(t)}, \nonumber
		\end{align}
		so that its time derivative generates expression like
		\begin{align} 
			+\frac{\nu+\lambda}{4} \norm{\partial_y^l \psi_{+,0}}_{L^2}^2
			-\frac{\nu+\lambda}{4M^2}\norm{\partial_y^{l+1}\eta_{+,0}}_{L^2}^2
			-\frac{\nu+\lambda}{4}4\pi \norm{\partial_y^{l+1}(-\partial_{yy}+4\pi M^2)^{-\frac{1}{2}}\eta_{+,0}}_{L^2}^2. \nonumber
		\end{align}
		The last two terms contribute to decay on the energy. However, the first term is undesired term. After carefully reviewing (\ref{natural energy}) and (\ref{derivative of natural energy}), we proceed with a second modification to prevent occurrence of such undesired term. Eventually, we get the energy functional $\mathcal{E}^l(t)$ in Theorem \ref{theorem k=0 ion} and its coercive functional $E^l(t)$, which will appear in the following proof.
	\end{remark}
	\begin{remark}
		\label{remark of ion 0 mode initial value}
		From Theorem \ref{theorem k=0 ion}, we infer that system of other modes, i.e. $(\eta_{+}-\eta_{+,0}, \psi_{+}-\psi_{+,0}, \omega_{+}-\omega_{+,0})$ is equivalent to system of $(\eta_{+}, \psi_{+}, \omega_{+})$ with $(\eta_{+,0}^{in}, \psi_{+,0}^{in}, \omega_{+,0}^{in})=(0,0,0)$, and we only consider the latter case when analyzing other modes.
	\end{remark}

	\begin{proof}
		(Proof of Theorem \ref{theorem k=0 ion}) 
		In order to apply Gr\"onwall's inequality, we define the energy functional, i.e.
		\begin{align}
			E^l(t):= \frac{1}{2} 
			\bigg(&
			\mathcal{E}^l(t) -  \frac{\nu+\lambda}{2}\lan{\partial_y^l \eta_{+,0}(t) , \partial_y^l \psi_{+,0}(t)}\nonumber 
			\bigg).
		\end{align}
		Due to $\nu+\lambda \leq 1, M^2(\nu +\lambda) \leq 1$, the functional $E^l(t)$ is coercive
		\begin{align}
			\frac{1}{4} \mathcal{E}^l(t) \leq E^l(t) \leq \mathcal{E}^l(t).\nonumber
		\end{align}
		Directly computing the time derivative of $E^l(t)$, we obtain
		\begin{align}
			\label{derivative of E^l(t)}
			\frac{\mathrm{d}}{\mathrm{dt}}E^l(t) 
			=&-(\nu+\lambda) \norm{\partial_y^{l+1} \psi_{+,0}}_{L^2}^2
			-(\nu+\lambda) \norm{\partial_y^{l}\psi_{+,0}}_{L^2}^2\nonumber \\
			&-\nu \norm{\partial_y^{l+1}(\omega_{+,0}+\eta_{+,0}-\nu M^2 \psi_{+,0})}_{L^2}^2 \nonumber \\
			&-4\pi \nu M^2 \lan{\partial_y^{l+1}(-\partial_{yy}+4\pi M^2)^{-1}\eta_{+,0},\partial_y^{l+1} (\omega_{+,0}+\eta_{+,0}-\nu M^2 \psi_{+,0} )} \nonumber \\
			&+\nu \lambda M^2 \lan{\partial_y^{l+1}\psi_{+,0}, \partial_y^{l+1} (\omega_{+,0}+\eta_{+,0}-\nu M^2 \psi_{+,0})}\nonumber \\
			&+\frac{\nu+\lambda}{4} \norm{\partial_y^l \psi_{+,0}}_{L^2}^2
			-\frac{\nu+\lambda}{4M^2}\norm{\partial_y^{l+1}\eta_{+,0}}_{L^2}^2\nonumber \\
			&-\frac{\nu+\lambda}{4}4\pi \norm{\partial_y^{l+1}(-\partial_{yy}+4\pi M^2)^{-\frac{1}{2}}\eta_{+,0}}_{L^2}^2\nonumber \\
			&+\frac{(\nu+\lambda)^2}{4} \lan{\partial_y^{l+1}\eta_{+,0}, \partial_y^{l+1}\psi_{+,0}}
		\end{align}
		In which, five negative sign dissipative terms, 
		\begin{align}
			&(\nu+\lambda)\norm{\partial_y^{l+1} \psi_{+,0}}_{L^2}^2, \:\:
			(\nu+\lambda) \norm{\partial_y^{l} \psi_{+,0}}_{L^2}^2, \:\:
			\frac{\nu+\lambda}{4M^2} \norm{\partial_y^{l+1}\eta_{+,0}}_{L^2}^2, \nonumber \\
			&\nu \norm{\partial_y^{l+1} (\omega_{+,0}+\eta_{+,0}-\nu M^2\psi_{+,0}) }_{L^2}^2 , \:\:
			\frac{\nu+\lambda}{4}4\pi \norm{\partial_y^{l+1} (-\partial_{yy}+4\pi M^2)^{-\frac{1}{2}}\eta_{+,0}}_{L^2}^2 ,\nonumber
		\end{align}
		can absorb all other terms in the right-hand side as follows. 
		\begin{align}
			\label{hard term 1 in derivative of E^l(t)}
			&-4\pi \nu M^2 \lan{\partial_y^{l+1}(-\partial_{yy}+4\pi M^2)^{-1}\eta_{+,0},\partial_y^{l+1} (\omega_{+,0}+\eta_{+,0}-\nu M^2 \psi_{+,0} )} \nonumber \\
			\leq& 2\sqrt{\pi}\nu M \norm{ \frac{(i\xi)^{l+1}}{(\xi^2+4\pi M^2)^{\frac{1}{2}}} \widehat{\eta_{+,0}}(\xi) }_{L^2} \norm{\partial_y^{l+1} (\omega_{+,0}+\eta_{+,0}-\nu M^2 \psi_{+,0} )}_{L^2}\nonumber \\
			\leq& \frac{\nu+\lambda}{8}4\pi \norm{\partial_y^{l+1}(-\partial_{yy}+4\pi M^2)^{-\frac{1}{2}}\eta_{+,0}}_{L^2}^2+\frac{1}{4}\nu\norm{\partial_y^{l+1} (\omega_{+,0}+\eta_{+,0}-\nu M^2 \psi_{+,0} )}_{L^2}^2,
		\end{align}
		where we have employed $0\leq \frac{16 \nu M^2}{\nu+\lambda}\leq 1 $. 
		\begin{align}
			\label{hard term 2 in derivative of E^l(t)}
			&\nu \lambda M^2 \lan{\partial_y^{l+1}\psi_{+,0}, \partial_y^{l+1} (\omega_{+,0}+\eta_{+,0}-\nu M^2 \psi_{+,0})} \nonumber \\
			\leq& \frac{\nu}{2} \norm{ \partial_y^{l+1}\psi_{+,0} }_{L^2}^2
			+\frac{\nu}{2} \norm{ \partial_y^{l+1} (\omega_{+,0}+\eta_{+,0}-\nu M^2 \psi_{+,0}) }_{L^2}^2,
		\end{align}
		where we have used $\lambda M^2\leq 1$, easily obtained from $M^2(\nu+\lambda)\leq1  $. 
		\begin{align}
			\label{hard term 3 in derivative of E^l(t)}
			&\frac{(\nu+\lambda)^2}{4} \lan{\partial_y^{l+1}\eta_{+,0}, \partial_y^{l+1}\psi_{+,0}}\nonumber \\
			\leq&\frac{\nu+\lambda}{8} \norm{\partial_y^{l+1}\psi_{+,0} }_{L^2}^2
			+ \xkh{M(\nu+\lambda)}^2\frac{\nu+\lambda}{8M^2}\norm{\partial_y^{l+1}\eta_{+,0} }_{L^2}^2.
		\end{align}
		Combining (\ref{derivative of E^l(t)})-(\ref{hard term 3 in derivative of E^l(t)}), we obtain
		\begin{align}
			\label{estimate of derivative of E^l(t)}
			\frac{\mathrm{d}}{\mathrm{dt}}E^l(t) 
			\leq& -\frac{3}{8}\nu \norm{\partial_y^{l+1} \psi_{+,0}}_{L^2}^2
			-\frac{3}{4}\nu \norm{\partial_y^{l}\psi_{+,0}}_{L^2}^2
			-\frac{\nu}{8} \frac{1}{M^2} \norm{\partial_y^{l+1}\eta_{+,0} }_{L^2}^2 \nonumber \\
			& -\frac{1}{4}\nu \norm{\partial_y^{l+1}(\omega_{+,0}+\eta_{+,0}-\nu M^2 \psi_{+,0})}_{L^2}^2  \nonumber \\
			&-\frac{\nu}{8} 4\pi \norm{\partial_y^{l+1}(-\partial_{yy}+4\pi M^2)^{-\frac{1}{2}}\eta_{+,0}}_{L^2}^2.
		\end{align}
		In account of the coercive of $E^l(t)$ and the above estimate, it is obvious that
		\begin{align}
			\label{mathcal E estimate}
			\mathcal{E}^l(t)\lesssim \mathcal{E}^l_{in}.
		\end{align}
		In particular, for $l=0$ we have
		\begin{align}
			\label{mathcal E estimate l=0}
			&\norm{ \psi_{+,0}(t)}^2_{L^2}
			+ \norm{\partial_y^{-1} \psi_{+,0}(t)}^2_{L^2} + \frac{1}{M^2}\norm{  \eta_{+,0}(t)  }^2_{L^2} +4\pi  \norm{ (-\partial_{yy}+4\pi M^2)^{-\frac{1}{2}}\eta_{+,0}(t)}_{L^2}^2  \nonumber \\
			& 
			+ \norm{ \xkh{\omega_{+,0} +\eta_{+,0} -\nu M^2 \psi_{+,0} }(t)}^2_{L^2}
			\lesssim \mathcal{E}^0_{in}.
		\end{align}
		For $l\geq 1$, applying interpolation inequality
		\begin{align}
			\norm{\partial_y^l f}_{L^2}
			\lesssim \norm{\partial_y^{l+1}f}_{L^2}^{\frac{l}{l+1}}
			\norm{f}_{L^2}^{\frac{1}{l+1}},
		\end{align}
		one has
		\begin{align}
			&\norm{\partial_y^{l+1} \psi_{+,0}}_{L^2}^2
			+\norm{\partial_y^{l}\psi_{+,0}}_{L^2}^2
			+\frac{1}{M^2} \norm{\partial_y^{l+1}\eta_{+,0} }_{L^2}^2 \nonumber \\
			& +\norm{\partial_y^{l+1}(\omega_{+,0}+\eta_{+,0}-\nu M^2 \psi_{+,0})}_{L^2}^2  
			+4\pi \norm{\partial_y^{l+1}(-\partial_{yy}+4\pi M^2)^{-\frac{1}{2}}\eta_{+,0}}_{L^2}^2 \nonumber \\
			\gtrsim& \norm{\partial_y^{l} \psi_{+,0}}_{L^2}^{2\frac{l+1}{l}}
			\norm{\psi_{+,0}}_{L^2}^{-2\frac{1}{l}}
			+ \norm{\partial_y^{l-1}\psi_{+,0}}_{L^2}^{2\frac{l+1}{l}}
			\norm{\partial_y^{-1}\psi_{+,0}}_{L^2}^{-2\frac{1}{l}}
			\nonumber\\
			&+ \norm{\frac{1}{M}\partial_y^l  \eta_{+,0}}_{L^2}^{2\frac{l+1}{l}}
			\norm{\frac{1}{M}\eta_{+,0}}_{L^2}^{-2\frac{1}{l}}+ \norm{\frac{1}{M}\partial_y^{l+1}\eta_{+,0}}_{L^2}^{2\frac{l+1}{l}} \norm{\frac{1}{M}\partial_y^{l+1}\eta_{+,0}}_{L^2}^{-\frac{2}{l}}
			\nonumber \\
			&+\norm{\partial_y^l (\omega_{+,0}+\eta_{+,0}-\nu M^2 \psi_{+,0})}_{L^2}^{2\frac{l+1}{l}}
			\norm{\omega_{+,0}+\eta_{+,0}-\nu M^2 \psi_{+,0}}_{L^2}^{-2\frac{1}{l}}
			\nonumber \\
			&+4\pi \norm{\partial_y^l (-\partial_{yy}+4\pi M^2)^{-\frac{1}{2}}\eta_{+,0}}_{L^2}^{2\frac{l+1}{l}}
			\norm{(-\partial_{yy}+4\pi M^2)^{-\frac{1}{2}}\eta_{+,0} }_{L^2}^{-2\frac{1}{l}} \nonumber \\
			& +\norm{\sqrt{4\pi}\partial_y^{l+1}(-\partial_{yy}+4\pi M^2)^{-\frac{1}{2}}\eta_{+,0} }_{L^2}^{2\frac{l+1}{l}} \norm{\sqrt{4\pi}\partial_y^{l+1}(-\partial_{yy}+4\pi M^2)^{-\frac{1}{2}}\eta_{+,0} }_{L^2}^{-\frac{2}{l}} \nonumber \\
			\gtrsim& \Bigg(
			\norm{\partial_y^{l} \psi_{+,0}}_{L^2}^{2\frac{l+1}{l}}
			+ \norm{\partial_y^{l-1}\psi_{+,0}}_{L^2}^{2\frac{l+1}{l}}
			+ \norm{\frac{1}{M}\partial_y^l \eta_{+,0}}_{L^2}^{2\frac{l+1}{l}}\nonumber \\
			&+\norm{\partial_y^l (\omega_{+,0}+\eta_{+,0}-\nu M^2 \psi_{+,0})}_{L^2}^{2\frac{l+1}{l}}
			+4\pi \norm{\partial_y^l (-\partial_{yy}+4\pi M^2)^{-\frac{1}{2}}\eta_{+,0}}_{L^2}^{2\frac{l+1}{l}}	
			\Bigg) (\mathcal{E}^o_{in})^{-\frac{1}{l}}\nonumber \\
			&+\Bigg(
			\norm{\frac{1}{M}\partial_y^{l+1}\eta_{+,0}}_{L^2}^{2\frac{l+1}{l}} 
			+
			\norm{\sqrt{4\pi}\partial_y^{l+1}(-\partial_{yy}+4\pi M^2)^{-\frac{1}{2}}\eta_{+,0} }_{L^2}^{2\frac{l+1}{l}}
			\Bigg) (\mathcal{E}^l_{in})^{-\frac{1}{l}} ,\nonumber
		\end{align}
		where in the last inequality we have used (\ref{mathcal E estimate}) and (\ref{mathcal E estimate l=0}). For $l\geq 1$, we get
		\begin{align}
			&\norm{\partial_y^{l+1} \psi_{+,0}}_{L^2}^2
			+\norm{\partial_y^{l}\psi_{+,0}}_{L^2}^2
			+\frac{1}{M^2} \norm{\partial_y^{l+1}\eta_{+,0} }_{L^2}^2 \nonumber \\
			& +\norm{\partial_y^{l+1}(\omega_{+,0}+\eta_{+,0}-\nu M^2 \psi_{+,0})}_{L^2}^2  
			+4\pi \norm{\partial_y^{l+1}(-\partial_{yy}+4\pi M^2)^{-\frac{1}{2}}\eta_{+,0}}_{L^2}^2 \nonumber \\
			\gtrsim& \Bigg(
			\norm{\partial_y^{l} \psi_{+,0}}_{L^2}^{2}
			+ \norm{\partial_y^{l-1}\psi_{+,0}}_{L^2}^{2}
			+ \frac{1}{M^2}\norm{\partial_y^l \eta_{+,0}}_{L^2}^{2}\nonumber \\
			&+\norm{\partial_y^l (\omega_{+,0}+\eta_{+,0}-\nu M^2 \psi_{+,0})}_{L^2}^{2}
			+4\pi \norm{\partial_y^l (-\partial_{yy}+4\pi M^2)^{-\frac{1}{2}}\eta_{+,0}}_{L^2}^{2}	
			\Bigg)^{\frac{l+1}{l}} (\mathcal{E}^o_{in})^{-\frac{1}{l}}\nonumber \\
			&+\Bigg(
			\frac{1}{M^2}\norm{\partial_y^{l+1}\eta_{+,0}}_{L^2}^{2} 
			+
			4\pi \norm{\partial_y^{l+1}(-\partial_{yy}+4\pi M^2)^{-\frac{1}{2}}\eta_{+,0} }_{L^2}^{2}
			\Bigg)^{\frac{l+1}{l}} (\mathcal{E}^l_{in})^{-\frac{1}{l}} \nonumber \\
			\geq & \: (\mathcal{E}^l(t))^{\frac{l+1}{l}} \:\: \xkh{\mathrm{max}\dkh{\mathcal{E}^0_{in}, \mathcal{E}^l_{in}}}^{-\frac{1}{l}}.\nonumber
		\end{align}
		Thanks to the coercive property of $E^l(t)$, combining the above estimate with (\ref{estimate of derivative of E^l(t)}), we infer 
		\begin{align}
			\frac{\mathrm{d}}{\mathrm{dt}}E^l(t)
			\leq -\nu C  \xkh{\mathrm{max}\dkh{\mathcal{E}^0_{in}, \mathcal{E}^l_{in}}}^{-\frac{1}{l}} ( E^l(t) )^{\frac{l+1}{l}},
		\end{align}
		where $C$ is a constant depending on $l$. Therefore, applying Gr\"onwall's inequality, we obtain
		\begin{align}
			E^l(t)\leq\dfrac{E^l_{in}}{\xkh{\nu C (E^l_{in})^\frac{1}{l}\xkh{\mathrm{max}\dkh{\mathcal{E}^0_{in}, \mathcal{E}^l_{in}}}^{-\frac{1}{l}} t+1   }^l}.
		\end{align}
		Applying the coercive property of $E^l(t)$ again, we conclude the proof of Theorem \ref{theorem k=0 ion}.
	\end{proof}

	\subsection{$k=0$ mode of Electron system}
	\ 
	
	As shown in ion case, we project (\ref{linear elec sys eta psi om 1})-(\ref{linear elec sys eta psi om 3}) onto $k=0$ frequency, we get
	\begin{align}
		\label{electron k=0 equ 1}
		&\partial_{t}\eta_{-,0}+\psi_{-,0}=0,\\
		\label{electron k=0 equ 2 }
		&\partial_{t}\psi_{-,0}+\frac{1}{M^2}\partial_{yy}\eta_{-,0}=4\pi \eta_{-,0}+(\nu+\lambda)\partial_{yy} \psi_{-,0},\\
		\label{electron k=0 equ 3 }
		&\partial_{t}\omega_{-,0}-\psi_{-,0}=\nu\partial_{yy}\omega_{-,0}.
	\end{align}
	Plugging (\ref{electron k=0 equ 1}) into (\ref{electron k=0 equ 2 }), we obtain
	\begin{align}
		\label{electron plug equ1 into equ2}
		-\partial_{tt}\eta_{-,0} +\frac{1}{M^2}\partial_{yy}\eta_{-,0}=4\pi \eta_{-,0}-(\nu+\lambda)\partial_{yy} \partial_{t}\eta_{-,0}.
	\end{align}
	When $\nu=\lambda=0$, taking inner product on both sides of (\ref{electron plug equ1 into equ2}) with $\partial_{t}\eta_{-,0}$, we have
	\begin{align}
		\frac{\mathrm{d}}{\mathrm{dt}}
		\xkh{
			\norm{\partial_{t}\eta_{-,0}}_{L^2}^2
			+ \frac{1}{M^2}\norm{\partial_y\eta_{-,0}}_{L^2}^2
			+4\pi \norm{\eta_{-,0}}_{L^2}^2
		}
		=0
	\end{align}
	It follows that $(\eta_{-,0},\psi_{-,0},\omega_{-,0})=(0,0,0)$, if $ (\eta_{-,0}^{in},\psi_{-,0}^{in},\omega_{-,0}^{in})=(0,0,0)$. For $\nu+\lambda>0$, we have the following theorem.
	\begin{theorem}
		\label{theorem electron k=0}
		Suppose $\nu, \lambda \geq 0 $, and $\eta_{-,0}^{in}, \psi_{-,0}^{in}, \omega_{-,0}^{in}$ be the initial data of (\ref{electron k=0 equ 1})-(\ref{electron k=0 equ 3 }). For any $ l \geq 0$, define
		\begin{align}
			\label{definition of mathcal F^l }
			\mathcal{F}^l (t)=&\norm{\partial_y^l \psi_{-,0}(t)}^2_{L^2}
			+ \norm{\partial_y^{l-1} \psi_{-,0}(t)}^2_{L^2} + \norm{\partial_y^{l-2} \psi_{-,0}(t)}^2_{L^2}\nonumber \\
			&+ \frac{1}{M^2}\xkh{ \norm{ \partial_y^{l+1} \eta_{-,0}(t) }^2_{L^2} + \norm{ \partial_y^l \eta_{-,0}(t)  }^2_{L^2} +\norm{ \partial_y^{l-1} \eta_{-,0}(t)  }^2_{L^2}} \nonumber \\
			& +4\pi \xkh{ \norm{\partial_y^l \eta_{-,0}}_{L^2}^2+ \norm{\partial_y^{l-1} \eta_{-,0}}_{L^2}^2
				+\norm{\partial_y^{l-2} \eta_{-,0}}_{L^2}^2}\nonumber \\
			& + \norm{\partial_y^l \xkh{\omega_{-,0} +\eta_{-,0} -\nu M^2 \psi_{-,0} }(t)}^2_{L^2} . 
		\end{align}
		If $0<\nu+\lambda \leq 1, M^2(\nu +\lambda) \leq 1 , 0\leq \frac{32 \pi\nu M^4}{\nu+\lambda}\leq 1  $, then the following decay estimate holds
		\begin{align}
			\mathcal{F}^l(t)\leq\dfrac{\widetilde{C} \mathcal{F}^l_{in}}{\xkh{\nu C_l  t+1   }^l} ,
		\end{align}
		where $C_l$ is a constant depending on $l$, with $C_0=0$, and $\widetilde{C}$ is a constant independent to $l$. 
	\end{theorem}
	\begin{remark}
		\label{remark on k=0 electron energy} 
		As discussed in Remark \ref{remark on k=0 ion energy}, we initially define 
		\begin{align}
			\mathcal{F}(t)=&\norm{\partial_y^l \psi_{-,0}(t)}^2_{L^2}
			+ \norm{\partial_y^{l-1} \psi_{-,0}(t)}^2_{L^2} 
			+ \frac{1}{M^2}\xkh{ \norm{ \partial_y^{l+1} \eta_{-,0}(t) }^2_{L^2} + \norm{ \partial_y^l \eta_{-,0}(t)  }^2_{L^2} } \nonumber \\
			& +4\pi \xkh{ \norm{\partial_y^l \eta_{-,0}}_{L^2}^2+ \norm{\partial_y^{l-1} \eta_{-,0}}_{L^2}^2 }
			+ \norm{\partial_y^l \xkh{\omega_{-,0} +\eta_{-,0} -\nu M^2 \psi_{-,0} }(t)}^2_{L^2} , \nonumber
		\end{align}
		and
		\begin{align}
			F(t):= \frac{1}{2} 
			\bigg(&
			\mathcal{F}^l(t) -  \frac{\nu+\lambda}{2}\lan{\partial_y^l \eta_{-,0}(t) , \partial_y^l \psi_{-,0}(t)} \nonumber 
			\bigg).
		\end{align}
		The time derivative of the last term in $\mathcal{F}(t)$ generates
		\begin{align}
			\label{bad term in energy derivative k=0 electron}
			4\pi \nu M^2 \lan{\partial_y^{l-1} \eta_{-,0}, \partial_y^{l+1}(\omega_{-,0}+\eta_{-,0}-\nu M^2\psi_{-,0})} .
		\end{align}
		Then we modify $F(t)$ as follow
		\begin{align}
			F(t):= \frac{1}{2} 
			\bigg(&
			\mathcal{F}^l(t) -  \frac{\nu+\lambda}{2}\lan{\partial_y^l \eta_{-,0}(t) , \partial_y^l \psi_{-,0}(t)}
			-  \frac{\nu+\lambda}{2}\lan{\partial_y^{l-1} \eta_{-,0}(t) , \partial_y^{l-1} \psi_{-,0}(t)} \nonumber 
			\bigg).
		\end{align}
		The time derivative of the last term generates expression
		\begin{align}
			+\frac{\nu+\lambda}{4}\norm{\partial_y^{l-1}\psi_{-,0}}_{L^2}^2
			-\frac{\nu+\lambda}{4}4\pi\norm{\partial_y^{l-1}\eta_{-,0}}_{L^2}^2. \nonumber
		\end{align}
		The latter term can absorbed $+\norm{\partial_y^{l-1}\eta_{-,0}}_{L^2}^2$ coming from (\ref{bad term in energy derivative k=0 electron}) after applying Cauchy inequality.
		For the former term, as shown in Remark \ref{remark on k=0 ion energy}, we need a second modification and we eventually obtain $\mathcal{F}^l(t)$ and $F^l(t)$.
	\end{remark}
	\begin{remark}
		\label{remark of electron 0 mode initial value}
		From Theorem \ref{theorem electron k=0}, we also infer that system of other modes, i.e. $(\eta_{-}-\eta_{-,0}, \psi_{-}-\psi_{-,0}, \omega_{-}-\omega_{-,0})$ is equivalent to system of $(\eta_{-}, \psi_{-}, \omega_{-})$ with $(\eta_{-,0}^{in}, \psi_{-,0}^{in}, \omega_{-,0}^{in})=(0,0,0)$, and we also only consider the latter case in the subsequent analysis.
	\end{remark}
	\begin{proof}(Proof of Theorem \ref{theorem electron k=0}) 
		We define the energy functional, i.e.
		\begin{align}
			F^l(t):= \frac{1}{2} 
			\bigg(&
			\mathcal{F}^l(t) -  \frac{\nu+\lambda}{2}\lan{\partial_y^l \eta_{-,0}(t) , \partial_y^l \psi_{-,0}(t)}
			-  \frac{\nu+\lambda}{2}\lan{\partial_y^{l-1} \eta_{-,0}(t) , \partial_y^{l-1} \psi_{-,0}(t)} \nonumber 
			\bigg).
		\end{align}
		Due to $\nu+\lambda \leq 1, M^2(\nu +\lambda) \leq 1$, the functional $F^l(t)$ is coercive
		\begin{align}
			\frac{1}{4} \mathcal{F}^l(t) \leq F^l(t) \leq \mathcal{F}^l(t).\nonumber
		\end{align}
		Directly computing the time derivative of $F^l(t)$, we get
		\begin{align}
			\label{derivative of F^l(t)}
			\frac{\mathrm{d}}{\mathrm{dt}}F^l(t) 
			=& -(\nu+\lambda) \norm{\partial_y^{l+1}\psi_{-,0}}_{L^2}^2
			-(\nu+\lambda) \norm{\partial_y^l\psi_{-,0}}_{L^2}^2
			-(\nu+\lambda) \norm{\partial_y^{l-1}\psi_{-,0}}_{L^2}^2 \nonumber \\
			&-\nu \norm{\partial_y^{l+1}(\omega_{-,0}+\eta_{-,0}-\nu M^2 \psi_{-,0})}_{L^2}^2 \nonumber \\
			&-4\pi \nu M^2 \lan{\partial_y^l \eta_{-,0}, \partial_y^l(\omega_{-,0}+\eta_{-,0}-\nu M^2\psi_{-,0})}\nonumber \\
			&+\nu \lambda M^2 \lan{\partial_y^{l+1}\psi_{-,0}, \partial_y^{l+1}(\omega_{-,0}+\eta_{-,0}-\nu M^2\psi_{-,0})}\nonumber \\
			&+\frac{\nu+\lambda}{4}\norm{\partial_y^l\psi_{-,0}}_{L^2}^2
			-\frac{\nu+\lambda}{4M^2}\norm{\partial_y^{l+1}\eta_{-,0}}_{L^2}^2\nonumber \\
			&-\frac{\nu+\lambda}{4}4\pi \norm{\partial_y^l\eta_{-,0}}_{L^2}^2
			+ \frac{(\nu+\lambda)^2}{4} \lan{\partial_y^{l+1}\eta_{-,0},\partial_y^{l+1}\psi_{-,0}}\nonumber \\
			&+\frac{\nu+\lambda}{4}\norm{\partial_y^{l-1}\psi_{-,0}}_{L^2}^2
			-\frac{\nu+\lambda}{4M^2}\norm{\partial_y^l \eta_{-,0}}_{L^2}^2\nonumber \\
			&-\frac{\nu+\lambda}{4}4\pi\norm{\partial_y^{l-1}\eta_{-,0}}_{L^2}^2
			+\frac{(\nu+\lambda)^2}{4}\lan{\partial_y^l\eta_{-,0},\partial_y^l \psi_{-,0}}.
		\end{align}
		Analogous to ion case, the dissipative terms can absorb the mixed terms and positive terms, but with different technical considerations. For this reason, we discuss the ion case and electron case separately in the analysis of $k=0$ mode.
		
		 From $0\leq \frac{32 \pi\nu M^4}{\nu+\lambda}\leq 1$ and integration by parts, we get
		\begin{align}
			&-4\pi \nu M^2 \lan{\partial_y^l \eta_{-,0}, \partial_y^l(\omega_{-,0}+\eta_{-,0}-\nu M^2\psi_{-,0})}\nonumber \\
			\leq& \frac{\nu+\lambda}{8} 4\pi \norm{\partial_y^{l-1} \eta_{-,0}}_{L^2}^2
			+ \frac{\nu}{4}\norm{ \partial_y^{l+1}(\omega_{-,0}+\eta_{-,0}-\nu M^2\psi_{-,0}) }_{L^2}^2 .\nonumber
		\end{align}
		It should be noted that, unlike (\ref{hard term 1 in derivative of E^l(t)}) in ion case, directly applying Cauchy inequality on this term does not work, as it produces  $\norm{ \partial_y^{l}(\omega_{-,0}+\eta_{-,0}-\nu M^2\psi_{-,0}) }_{L^2}^2$, which eventually leads to $\norm{ \partial_y^{l-1}(\omega_{-,0}+\eta_{-,0}-\nu M^2\psi_{-,0}) }_{L^2}^2$ after using interpolation  inequality. But there is no such term in $\mathcal{F}^l(t)$, making it difficult to apply Gr\"onwall's inequality in subsequent analysis. 
		
		We bound other mixed terms in the right-hand side of (\ref{derivative of F^l(t)}) as follows. By $M^2(\nu +\lambda) \leq 1$ and Cauchy inequality, we obtain
		\begin{align}
			&\nu \lambda M^2 \lan{\partial_y^{l+1}\psi_{-,0}, \partial_y^{l+1}(\omega_{-,0}+\eta_{-,0}-\nu M^2\psi_{-,0})}\nonumber \\
			\leq& \frac{\nu}{2} \norm{\partial_y^{l+1}\psi_{-,0}}_{L^2}^2
			+\frac{\nu}{2} \norm{\partial_y^{l+1}(\omega_{-,0}+\eta_{-,0}-\nu M^2 \psi_{-,0})}_{L^2}^2,\nonumber
		\end{align}
		and
		\begin{align}
			&\frac{(\nu+\lambda)^2}{4} \lan{\partial_y^{l+1}\eta_{-,0}, \partial_y^{l+1}\psi_{-,0}}\nonumber \\
			\leq & \frac{\nu+\lambda}{8} \norm{\partial_y^{l+1}\psi_{-,0} }_{L^2}^2
			+ \frac{\nu+\lambda}{8M^2}\norm{\partial_y^{l+1}\eta_{-,0} }_{L^2}^2 .\nonumber
		\end{align}
		Similarly, we have
		\begin{align}
			&\frac{(\nu+\lambda)^2}{4} \lan{\partial_y^{l}\eta_{-,0}, \partial_y^{l}\psi_{-,0}}\nonumber \\
			\leq & \frac{\nu+\lambda}{8} \norm{\partial_y^{l}\psi_{-,0} }_{L^2}^2
			+ \frac{\nu+\lambda}{8M^2}\norm{\partial_y^{l}\eta_{-,0} }_{L^2}^2 .\nonumber
		\end{align}
		Therefore, we have
		\begin{align}
			\label{estimate derivative of F^l(t)}		
			\frac{\mathrm{d}}{\mathrm{dt}}F^l(t) 
			\leq& -\frac{\nu}{8} \Bigg(
			\norm{\partial_y^{l+1}\psi_{-,0}}_{L^2}^2
			+\norm{\partial_y^l\psi_{-,0}}_{L^2}^2
			+\norm{\partial_y^{l-1}\psi_{-,0}}_{L^2}^2 \nonumber \\
			&+4\pi \norm{\partial_y^l\eta_{-,0}}_{L^2}^2
			+4\pi \norm{\partial_y^{l-1}\eta_{-,0}}_{L^2}^2
			+\frac{1}{M^2}\norm{\partial_y^{l+1}\eta_{-,0}}_{L^2}^2
			+\frac{1}{M^2}\norm{\partial_y^l\eta_{-,0}}_{L^2}^2 \nonumber \\
			&+\norm{ 		 \partial_y^{l+1}(\omega_{-,0}+\eta_{-,0}-\nu M^2\psi_{-,0}) }_{L^2}^2 
			\Bigg) .
		\end{align}
		In account of the coercive of $F^l(t)$ and the above estimate, we have
		\begin{align}
			\mathcal{F}^l(t)\lesssim \mathcal{F}^l_{in}.
		\end{align}
		In particular, for $l=0$ we have
		\begin{align}
			&\norm{ \psi_{-,0}(t)}^2_{L^2}
			+ \norm{\partial_y^{-1} \psi_{-,0}(t)}^2_{L^2} 
			+\norm{\partial_y^{-2} \psi_{-,0}(t)}^2_{L^2}\nonumber \\
			&+ 4\pi \norm{\partial_y^{-1}\eta_{-,0}}^2_{L^2}
			+ 4\pi \norm{\partial_y^{-2}\eta_{-,0}}^2_{L^2} 
			+ \frac{1}{M^2}\norm{  \eta_{-,0}(t)  }^2_{L^2}  
			+\frac{1}{M^2}\norm{ \partial_y^{-1} \eta_{-,0}(t)  }^2_{L^2} \nonumber \\
			& 
			+ \norm{ \xkh{\omega_{+,0} +\eta_{+,0} -\nu M^2 \psi_{+,0} }(t)}^2_{L^2}
			\lesssim \mathcal{F}^0_{in}.
		\end{align}
		For $l\geq 1$, applying the above estimate and interpolation inequality as shown in ion case,
		\begin{align}
			\norm{\partial_y^l f}_{L^2}
			\lesssim \norm{\partial_y^{l+1}f}_{L^2}^{\frac{l}{l+1}}
			\norm{f}_{L^2}^{\frac{1}{l+1}},\nonumber
		\end{align}
		one has
		\begin{align}
			&\norm{\partial_y^{l+1}\psi_{-,0}}_{L^2}^2
			+\norm{\partial_y^l\psi_{-,0}}_{L^2}^2
			+\norm{\partial_y^{l-1}\psi_{-,0}}_{L^2}^2 \nonumber \\
			&+4\pi \norm{\partial_y^l\eta_{-,0}}_{L^2}^2
			+4\pi \norm{\partial_y^{l-1}\eta_{-,0}}_{L^2}^2
			+\frac{1}{M^2}\norm{\partial_y^{l+1}\eta_{-,0}}_{L^2}^2
			+\frac{1}{M^2}\norm{\partial_y^l\eta_{-,0}}_{L^2}^2 \nonumber \\
			&+\norm{ 		 \partial_y^{l+1}(\omega_{-,0}+\eta_{-,0}-\nu M^2\psi_{-,0}) }_{L^2}^2 \nonumber \\
			\gtrsim&
			\Bigg(
			\norm{\partial_y^{l}\psi_{-,0}}_{L^2}^2
			+\norm{\partial_y^{l-1}\psi_{-,0}}_{L^2}^2
			+\norm{\partial_y^{l-2}\psi_{-,0}}_{L^2}^2 \nonumber \\
			&+4\pi \norm{\partial_y^{l-1}\eta_{-,0}}_{L^2}^2
			+4\pi \norm{\partial_y^{l-2}\eta_{-,0}}_{L^2}^2
			+\frac{1}{M^2}\norm{\partial_y^{l}\eta_{-,0}}_{L^2}^2
			+\frac{1}{M^2}\norm{\partial_y^{l-1}\eta_{-,0}}_{L^2}^2 \nonumber \\
			&+\norm{ 		 \partial_y^{l}(\omega_{-,0}+\eta_{-,0}-\nu M^2\psi_{-,0}) }_{L^2}^2 
			\Bigg) ^{\frac{l+1}{l}}  (\mathcal{F}^0_{in})^{-\frac{1}{l}}\nonumber \\
			\geq& (\mathcal{F}^l(t))^{\frac{l+1}{l}}(\mathcal{F}^0_{in})^{-\frac{1}{l}} . \nonumber
		\end{align}
		Therefore, we have
		\begin{align}
			\frac{\mathrm{d}}{\mathrm{dt}}F^l(t)
			\leq -\nu C (F^l(t)) ^{\frac{l+1}{l}} (\mathcal{F}^0_{in})^{-\frac{1}{l}},
		\end{align}
		where $C$ is a constant depending on $l$. Applying Gr\"onwall's inequality, we obtain
		\begin{align}
			F^l(t)\leq\dfrac{F^l_{in}}{\xkh{\nu C (F^l_{in})^\frac{1}{l} (\mathcal{F}^0_{in})^{-\frac{1}{l}} t+1   }^l}.
		\end{align}
		Thus, we end the proof of Theorem \ref{theorem electron k=0}.
	\end{proof}

\section{\textbf{Analysis of nonzero modes}}
		In this section, we study the system (\ref{linear sys eta psi om 1})-(\ref{linear sys eta psi om 3}) and (\ref{linear elec sys eta psi om 1})-(\ref{linear elec sys eta psi om 3}) of other modes ($k\neq 0$). In view of Remark \ref{remark of ion 0 mode initial value} and Remark \ref{remark of electron 0 mode initial value}, it is reasonable to assume $(\eta_{\pm,0}^{in}, \psi_{\pm,0}^{in}, \omega_{\pm,0}^{in})=(0,0,0)$ in this section. We will use the analogous method in \cite{Antonelli 2021}, which proved the linear stability of Couette flow for 2D Navier-Stokes systems. 
	The energy functional in our paper is different to the one in \cite{Antonelli 2021}, due to the existence of the self-consistent electric field in NSP.
	
	Rewrite (\ref{linear sys eta psi om 1})-(\ref{linear sys eta psi om 3}) in the new frame in terms of ($\Pi, \Psi, \Gamma$) as follows
	\begin{align} 
		&\partial_{t} \Pi_{+} + \Psi_{+}=0,\\
		&\begin{aligned}[t]
			\partial_{t} \Psi_{+}  
			+   2\partial_{X}(\partial_{Y}-&t\partial_{X}) \Delta_{L}^{-1} \Psi_{+}
			+   2\partial_{XX}\Delta_{L}^{-1}\Gamma_{+}
			+   \frac{1}{M^2}\Delta_{L}\Pi_{+} \nonumber \\
			&=   -4\pi \Delta_{L}(-\Delta_{L}+4\pi M^2)^{-1}\Pi_{+}
			+\mu\Delta_{L}\Psi_{+} , 
	 	 \end{aligned} \\
		&\partial_{t}\Gamma_{+}-\Psi_{+}= \nu \Delta_{L} \Gamma_{+} ,
	\end{align}
	where $\mu := \lambda +\nu$.  Taking the Fourier transform of the above system , we have
	\begin{align}
		\label{equ 1 after Fouier}
		&\partial_{t} \widehat{\Pi}_{+}=- \widehat{\Psi}_{+} ,\\
		\label{equ 2 after Fourier}
		&\partial_{t} \widehat{\Psi}_{+}
		= \frac{\partial_{t} \alpha}{\alpha} \widehat{\Psi}_{+}
		- \mu \alpha \widehat{\Psi}_{+} 
		+ \frac{1}{M^2} \alpha \widehat{\Pi}_{+} 
		- \frac{2k^2}{\alpha} \widehat{\Gamma}_{+}
		+ 4 \pi \frac{\alpha}{\alpha+ 4 \pi M^2} \widehat{\Pi}_{+} , \\
		\label{equ 3 after Fourier}
		&\partial_{t} \widehat{\Gamma}_{+} - \widehat{\Psi}_{+} =-\nu \alpha 		\widehat{\Gamma}_{+} . 
	\end{align}
	Applying the same techniques to (\ref{linear elec sys eta psi om 1})-(\ref{linear elec sys eta psi om 3}), we get
	\begin{align}
		\label{electron Fourier equ1}
		&\partial_{t} \widehat{\Pi}_{-}=- \widehat{\Psi}_{-} , \\
		&\partial_{t} \widehat{\Psi}_{-}
		= \frac{\partial_{t} \alpha}{\alpha} \widehat{\Psi}_{-}
		- \mu \alpha \widehat{\Psi}_{-} 
		+ \frac{1}{M^2} \alpha \widehat{\Pi}_{-} 
		- \frac{2k^2}{\alpha} \widehat{\Gamma}_{-}
		+ 4 \pi\widehat{\Pi}_{-} , \\
		\label{electron Fourier equ3}
		&\partial_{t} \widehat{\Gamma}_{-} - \widehat{\Psi}_{-} =-\nu \alpha 		\widehat{\Gamma}_{-} . 
	\end{align}
	By introducing parameter $\delta \in \{ 0, 1\}$, the above two systems can be represented into the following unified from,
	\begin{align}
		\label{unified form equ 1}
		&\partial_{t} \widehat{\Pi}=- \widehat{\Psi} ,\\
		\label{unified form equ 2}
		&\partial_{t} \widehat{\Psi}
		= \frac{\partial_{t} \alpha}{\alpha} \widehat{\Psi}
		- \mu \alpha \widehat{\Psi} 
		+ \frac{1}{M^2} \alpha \widehat{\Pi} 
		- \frac{2k^2}{\alpha} \widehat{\Gamma}
		+ 4 \pi \frac{\alpha}{\alpha+ 4 \pi M^2 \delta} \widehat{\Pi} , \\
		\label{unified form equ 3}
		&\partial_{t} \widehat{\Gamma} - \widehat{\Psi} =-\nu \alpha 		\widehat{\Gamma} ,
	\end{align}
	with $\delta=1$ corresponding to ion system, i.e.  $(\Pi,\Psi,\Gamma)=(\Pi_{+},\Psi_{+},\Gamma_{+})$, and $\delta=0$ corresponding to electron system, i.e. $(\Pi,\Psi,\Gamma)=(\Pi_{-},\Psi_{-},\Gamma_{-})$.
	%%%%%
	%%%%%

	%%%%%%%%%%%%%%%%%%%%%%%%%%%%%%%%%%%%%%%%%%%%%%%%%%%%%%%%%
	%%%%%%%%%%%%%%%%%%%%%%%%%%%%%%%%%%%%%%%%%%%%%%%%%%%%%%%%%
	%%%%%%%%%%%%%%%%%%%%%%%%%%%%%%%%%%%%%%%%%%%%%%%%%%%%%%%%%
	\subsection{Dissipation with loss of derivatives}
	\ 
	
	In this subsection, we will prove Theorem \ref{th1}, based on the following two key estimates.
	\begin{proposition} \label{proposition for th 1}
		Let $s\geq0$, $0<\mu \leq \frac{1}{2}$, $M>0$ be such that $M\leq\mathrm{min}\{\mu^{-1},\lambda^{-\frac{1}{2}},\nu^{-\frac{1}{3}}\}$. If $\eta^{in}_{\pm}\in  H^{s+1}(\mathbb{T}\times\mathbb{R})$ , $\psi^{in}_{\pm},\omega^{in}_{\pm} \in  H^s(\mathbb{T}\times\mathbb{R})$ and $\eta^{in}_{\pm,0}=\psi^{in}_{\pm,0}=\omega^{in}_{\pm,0}=0$, then
		\begin{align}
			\frac{1}{M}\norm{\xkh{\alpha^{-\frac{1}{4}}\widehat{\Pi}_{\pm}}(t)}_{H^s}&+\norm{\xkh{\alpha^{-\frac{3}{4}}\widehat{\Psi}_{\pm}}(t)}_{H^s}\nonumber \\
			&+\norm{\xkh{\alpha^{-\frac{3}{4}}(\widehat{F}_{\pm}-\nu M^2\widehat{\Psi}_{\pm})}(t)}_{H^s} 
			\lesssim \mathrm{e}^{-\frac{1}{32}\nu^\frac{1}{3}t} C_{in,s}^{\pm}.
		\end{align}
	\end{proposition}
	\begin{corollary} \label{corollary for th 1}
		Let $s\geq0$, $0<\mu \leq \frac{1}{2}$, $M>0$ be such that $M\leq\mathrm{min}\{\mu^{-1},\lambda^{-\frac{1}{2}},\nu^{-\frac{1}{3}}\}$. If $\eta^{in}_{\pm}\in H^{s+\frac{7}{2}}(\mathbb{T}\times\mathbb{R})$, $\psi^{in}_{\pm}, \omega^{in}_{\pm} \in H^{s+\frac{5}{2}}(\mathbb{T}\times\mathbb{R})$ and $\eta^{in}_{\pm,0}=\psi^{in}_{\pm,0}=\omega^{in}_{\pm,0}=0$, then
		\begin{align}
			\norm{\widehat{\Gamma}_{\pm}(t)}_{H^s} \lesssim&\:M \lan{t}^\frac{1}{2} \mathrm{e}^{-\frac{1}{32}\nu^\frac{1}{3}t} C^{\pm}_{in,s+\frac{1}{2}} + 
			M \lan{t}^\frac{1}{2} \mathrm{e}^{-\frac{1}{64}\nu^\frac{1}{3}t} C^{\pm}_{in,s+\frac{5}{2}} \nonumber \\
			&+\mathrm{e}^{-\frac{1}{12}\nu^\frac{1}{3}t} \norm{\omega^{in}_{\pm}+\eta^{in}_{\pm}}_{H^s}.
		\end{align}
	\end{corollary}
	\begin{remark}
		There are loss of derivatives in Proposition \ref{proposition for th 1} and Corollary \ref{corollary for th 1}. We will analyze the reason in Remark \ref{remark on loss of derivative}.
	\end{remark}
	
	In the following, we first prove Theorem (\ref{th1}) utilizing Proposition (\ref{proposition for th 1}) and Corollary (\ref{corollary for th 1}), and the proof of the proposition and the corollary are postponed to the end of this subsection.
	\begin{proof}
		(Proof of Theorem \ref{th1}) From the Helmholtz decomposition, we have
		\begin{align}
			&\norm{\xkh{\mathbb{Q}[u_{\pm}]}(t)}_{L^2}^2+\frac{1}{M^2}\norm{\eta_{\pm}(t)}_{L^2}^2 \nonumber \\
			=&\norm{\nabla \Delta^{-1} \psi_{\pm}(t) }_{L^2}^2+\frac{1}{M^2}\norm{\eta_{\pm}(t)}_{L^2}^2 \nonumber \\
			=&\norm{(-\Delta_L)^{-\frac{1}{2}}\Psi_{\pm}(t)}_{L^2}^2+\frac{1}{M^2}\norm{\Pi_{\pm}(t)}_{L^2}^2 \nonumber
		\end{align}
		Apply the Plancherel's Theorem and the fact that $\alpha(t,k,\xi) \lesssim \lan{t}^2 \lan{k,\xi}^2$, we obtain
		\begin{align}
			&\norm{\xkh{\mathbb{Q}[u_{\pm}]}(t)}_{L^2}^2+\frac{1}{M^2}\norm{\eta_{\pm}(t)}_{L^2}^2 \nonumber \\
			\lesssim &\lan{t} \xkh{ \norm{\xkh{\alpha^{-\frac{3}{4}}\widehat{\Psi}_{\pm}}(t)}^2_{H^\frac{1}{2}} +\frac{1}{M^2} \norm{\xkh{\alpha^{-\frac{1}{4}}\widehat{\Pi}_{\pm}}(t)}^2_{H^\frac{1}{2}}  } \nonumber \\
			\lesssim &\lan{t} \mathrm{e}^{-\frac{1}{16}\nu^\frac{1}{3}t} \xkh{C^{\pm}_{in,\frac{1}{2}}} ^2 , \nonumber 
		\end{align}
		where we have used Proposition \ref{proposition for th 1} in the last line, hence proving (\ref{conclusion1 in th 1}). By using the Helmholtz decomposition again, we get
		\begin{align}
			\norm{\mathbb{P}[u_{\pm}]^x(t)}_{L^2}=\norm{-\partial_y\Delta^{-1}\omega_{\pm}(t)}_{L^2}
			\leq \norm{\xkh{\frac{1}{\alpha^\frac{1}{2}}\widehat{\Gamma}_{\pm}}(t)}_{L^2}. \nonumber
		\end{align}
		For $k \neq 0$, applying $\lan{k,t} \lesssim \lan{\frac{\xi}{k}-t} \lan{k,\frac{\xi}{k}}$ , we get
		\begin{align}
			\frac{1}{\alpha}
			\lesssim  \frac{\lan{k,\xi}^2}{\lan{t}^2}.\nonumber
		\end{align}
		For $k=0$, we have assumed $\omega^{in}_{+,0}=0$. Thus, combining the above inequality with Corollary (\ref{corollary for th 1}), we deduce
		\begin{align}
			&\norm{\mathbb{P}[u_{\pm}]^x(t)}_{L^2} \lesssim \norm{\frac{\lan{k,\xi}}{\lan{t}}\widehat{\Gamma}_{\pm}(t)}_{L^2} =\frac{1}{\lan{t}}\norm{\widehat{\Gamma}_{\pm}(t)}_{H^1} \nonumber \\
			&\lesssim \frac{1}{\lan{t}} \xkh{  M \lan{t}^\frac{1}{2} \mathrm{e}^{-\frac{1}{32}\nu^\frac{1}{3}t} C^{\pm}_{in,\frac{3}{2}} + 
				M \lan{t}^\frac{1}{2} \mathrm{e}^{-\frac{1}{64}\nu^\frac{1}{3}t} C^{\pm}_{in,\frac{7}{2}} +\mathrm{e}^{-\frac{1}{12}\nu^\frac{1}{3}t} \norm{\omega^{in}_{\pm}+\eta^{in}_{\pm}}_{H^1}   } \nonumber \\
			&\lesssim M\frac{\mathrm{e}^{-\frac{1}{64}\nu^\frac{1}{3}t}}{\lan{t}^\frac{1}{2}} C^{\pm}_{in}
			+ \frac{\mathrm{e}^{-\frac{1}{12}\nu^\frac{1}{3}t}}{\lan{t}}\norm{\omega^{in}_{\pm}+\eta^{in}_{\pm}}_{H^1}. \nonumber
		\end{align}
		Applying $\lan{k,t} \lesssim \lan{\frac{\xi}{k}-t} \lan{k,\frac{\xi}{k}}$ again, we have
		\begin{align}
			\frac{\abs{k}}{\alpha}\lesssim \frac{\lan{k,\xi}^2}{\lan{t}^2} ,\quad \forall k \in \mathbb{Z}. \nonumber
		\end{align}
		Hence, applying Corollary \ref{corollary for th 1} again, we obtain
		\begin{align}
			&\norm{\mathbb{P}[u_{\pm}]^y(t)}_{L^2}= \norm{\partial_x\Delta^{-1}\omega_{\pm}(t)}_{L^2}=\norm{\partial_{X}\Delta_{L}^{-1}\Gamma_{\pm}(t)}_{L^2}  		
			=\norm{\frac{k}{\alpha} \widehat{\Gamma}_{\pm}(t)}_{L^2} \lesssim \frac{1}{\lan{t}^2} \norm{\widehat{\Gamma}_{\pm}(t)}_{H^2} \nonumber \\
			&\lesssim \frac{1}{\lan{t}^2} \xkh{ M \lan{t}^\frac{1}{2} \mathrm{e}^{-\frac{1}{32}\nu^\frac{1}{3}t} C^{\pm}_{in,\frac{5}{2}} + 
				M \lan{t}^\frac{1}{2} \mathrm{e}^{-\frac{1}{64}\nu^\frac{1}{3}t} C^{\pm}_{in,\frac{9}{2}} +\mathrm{e}^{-\frac{1}{12}\nu^\frac{1}{3}t} \norm{\omega^{in}_{\pm}+\eta^{in}_{\pm}}_{H^2} } \nonumber \\
			&\leq M \frac{\mathrm{e}^{-\frac{1}{64}\nu^\frac{1}{3}t}}{\lan{t}^\frac{3}{2}}C^{\pm}_{in}+\frac{\mathrm{e}^{-\frac{1}{12}\nu^\frac{1}{3}t}}{\lan{t}^2} \norm{\omega^{in}_{\pm}+\eta^{in}_{\pm}}_{H^2} . \nonumber
		\end{align}
		We conclude the proof of Theorem \ref{th1}.
	\end{proof}
	In order to prove Proposition \ref{proposition for th 1}, we need to define a suitable weighted energy functional and the proof of Proposition \ref{proposition for th 1} follows from the bounds on this energy functional.
	
	Introduce the following Fourier multiplier
	\begin{align}
		\partial_{t} m(t,k,\xi)&=\frac{2\nu^\frac{1}{3}}{\nu^\frac{2}{3}\xkh{\frac{\xi}{k}-t}^2+1} m(t,k,\xi), \nonumber \\
		m(0,k,\xi)&=\mathrm{e}^{2\mathrm{arctan}\xkh{-\nu^\frac{1}{3}\frac{\xi}{k}}}, \nonumber
	\end{align}
	which has been used in \cite{Antonelli 2021} and is explicitly given by
	\begin{align}
		\label{def of m}
		m(t,k,\xi)=\mathrm{e}^{2\mathrm{arctan}\xkh{\nu^\frac{1}{3}\xkh{t-\frac{\xi}{k}}}},
	\end{align}
	and satisfies the following crucial property
	\begin{align}
		\label{property about m}
		\nu \alpha (t,k,\xi)+\frac{\partial_{t}m}{m}(t,k,\xi)\geq \nu^\frac{1}{3},\quad \forall \: t\geq0,\: k\in \mathbb{Z} \backslash \{0\},\:\xi \in \mathbb{R}.
	\end{align}
	It's clear that $m$ and $m^{-1}$ are both bounded Fourier multiplier, therefore they generate equivalent norm to the standard $L^2$ norm.
	
	Let $s\geq 0$, and $\Pi, \Psi, \Gamma$ satisfy (\ref{unified form equ 1})-(\ref{unified form equ 3}). We define the following weighted variables
	\begin{align}
		&Z_1(t)=\frac{1}{M} \langle k,\xi \rangle ^s \xkh{m^{-1} \alpha^{-\frac{1}{4} }\widehat{\Pi}}(t) ,  \quad
		Z_2(t)=\langle k,\xi \rangle ^s  \xkh{m^{-1} \alpha^{-\frac{3}{4} }\widehat{\Psi}}(t) , \\
		&Z_3(t)=\langle k,\xi \rangle ^s \xkh{m^{-1} \alpha^{-\frac{3}{4} }(\widehat{F}-\nu M^2\widehat{\Psi})} (t) ,
	\end{align}
	where $F=\Pi+\Gamma$. 
	\begin{remark}
		\label{remark on definition of Z_1 Z_2 Z_3}
		Here, $m$ is lower and upper bounded multiplier. $\alpha^{-\frac{1}{4}}$ for $Z_1$ and $\alpha^{-\frac{3}{4}}$ for $Z_2$ are used in \cite{Pu and Zhou and Bian} to symmetrize the Euler-Poisson system. For technical reason, we choose $\alpha^{-\frac{3}{4}}$ for $Z_3$. In detail, the equation satisfied by ($\widehat{F}-\nu M^2 \widehat{\Psi}$) is  
		\begin{align}
			\label{equation satisfied by F- nu M^2 Psi}
			\partial_{t}\xkh{\widehat{F}-\nu M^2 \widehat{\Psi}}
			=&-\nu \alpha \xkh{\widehat{F}-\nu M^2 \widehat{\Psi}}
			+ \nu (\mu-\nu) M^2 \alpha \widehat{\Psi}
			-\nu M^2 \frac{\partial_{t}\alpha}{\alpha} \widehat{\Psi} \nonumber \\
			&+2\nu M^2\frac{k^2}{\alpha}\xkh{ \widehat{F}-\nu M^2 \widehat{\Psi}  }
			-2\nu M^2 \frac{k^2}{\alpha} \widehat{\Pi} \nonumber \\
			&+2\nu^2 M^4 \frac{k^2}{\alpha} \widehat{\Psi}
			-\nu M^2  \frac{4\pi \alpha}{\alpha+4\pi M^2 \delta} \widehat{\Pi}
		\end{align}
		Time derivative of $\abs{\alpha^j \xkh{\widehat{F}-\nu M^2 \widehat{\Psi}} }^2$ formally generates the expression 
		\begin{align}
			-\nu \alpha \abs{\alpha^j \xkh{\widehat{F}-\nu M^2 \widehat{\Psi}} }^2
			+ \nu (\mu-\nu) M^2 \alpha^{j+1+\frac{3}{4}} \mathrm{Re} \xkh{ \alpha^{-\frac{3}{4}}\Psi \alpha^j \overline{ \xkh{  \widehat{F}-\nu M^2 \widehat{\Psi} } }}. \nonumber
		\end{align}
		In order to use the first term to control the second term, it is required that $j \leq -\frac{3}{4}$. For this reason, we choose $\alpha^{-\frac{3}{4}}$ for $Z_3$.
	\end{remark}

	Let $0< \gamma:=\gamma(M,\nu)\leq \frac{1}{4}$ be a parameter which will be chosen subsequently and consider the following energy functional
	\begin{align}
		E_\delta(t)=\frac{1}{2} \bigg(  &\left(   1+ M^2\frac{(\partial_{t}\alpha)^2}{\alpha^3} + M^2\frac{4\pi}{\alpha+4\pi M^2\delta} \right) \vert Z_1 \vert ^2 +\vert Z_2 \vert ^2   +\vert Z_3 \vert ^2 \nonumber \\
		\label{mixed term in energy}
		&+\frac{M}{2}\frac{\partial_{t}\alpha}{\alpha ^\frac{3}{2}} \mathrm{Re}(\bar{Z}_1Z_2)  - 2\gamma\alpha^{-\frac{1}{2} } \mathrm{Re}(\bar{Z}_1Z_2)   \bigg)  (t) . \nonumber
	\end{align}
	Here, $E_{\delta=1}(t)$, $E_{\delta=0}(t)$ are energy functional of ion system and electron system, respectly. It is easily obtained that
	$E_\delta(t)$ is coercive, namely
	\begin{align}
		&\abs{ E_\delta(t) }\leq \xkh{\bigg( 1+M^2\frac{(\partial_{t}\alpha)^2}{\alpha^3} + 	M^2\frac{4\pi}{\alpha+4\pi M^2 \delta}     \bigg) \vert Z_1\vert^2 +\vert Z_2\vert^2 +\vert Z_3\vert^2   }(t) ,\\
		\label{coercive 2}
		&\abs{ E_\delta(t) }\geq \frac{1}{4} \bigg(   \left( 	1+M^2\frac{(\partial_{t}\alpha)^2}{\alpha^3} + M^2\frac{4\pi}{\alpha+4\pi M^2 \delta} \right)  \vert Z_1\vert^2 +\vert Z_2\vert^2 +2\vert Z_3\vert^2  \bigg) (t) .
	\end{align}
	Since $m$ is a bounded Fourier multiplier, we have
	\begin{align}
		\label{energy approx}
		\sum_{k\neq0}\int_{\mathbb{R}}E_\delta(t)d\xi \approx \frac{1}{M^2} 	\norm{\xkh{\alpha^{-\frac{1}{4}}\widehat{\Pi}}(t)}  _{H^s}^2&+ \norm{\xkh{\alpha^{-\frac{3}{4}} \widehat{\Psi}}(t)}_{H^s}^2 \nonumber \\
		&+ \norm{\xkh{\alpha^{-\frac{3}{4}} (\widehat{F}-\nu M^2\widehat{\Psi})}(t)}_{H^s}^2 ,
	\end{align}
	where we have used 
	\begin{align}
		1\leq 1+M^2 \frac{(\partial_{t}\alpha)^2}{\alpha^3} + 	M^2\frac{4\pi}{\alpha+4\pi M^2 \delta} \leq 1+(4\pi+1)M^2 ,\nonumber
	\end{align}
	namely
	\begin{align}
		1\leq 1+M^2 \frac{(\partial_{t}\alpha)^2}{\alpha^3} + 	M^2\frac{4\pi}{\alpha+4\pi M^2 \delta} \lesssim 1 .\nonumber
	\end{align}
	
	\begin{lemma} \label{lemma 1}
		Suppose $s\geq0$, $0<\mu \leq \frac{1}{2}$, $M>0$ satisfy $M\leq\mathrm{min}\{\mu^{-1},\lambda^{-\frac{1}{2}},\nu^{-\frac{1}{3}}\}$ and $\gamma:=\frac{M\nu^\frac{1}{3}}{4}$. If $\eta^{in}_{\pm}\in  H^{s+1}(\mathbb{T}\times\mathbb{R})$ , $\psi^{in}_{\pm},\omega^{in}_{\pm} \in  H^s(\mathbb{T}\times\mathbb{R})$ and $\eta^{in}_{\pm,0}=\psi^{in}_{\pm,0}=\omega^{in}_{\pm,0}=0$, then
		\begin{align}
			\label{bound 1 in lemma 1}
			\sum_{k\neq0}\int_{\mathbb{R}}E_{\delta=1}(t)d\xi \lesssim \mathrm{e}^{-\frac{\nu^\frac{1}{3}}{16}t}  (C^{+}_{in,s})^2 , \\
			\label{bound 2 in lemma 1}
			\sum_{k\neq0}\int_{\mathbb{R}}E_{\delta=0}(t)d\xi \lesssim \mathrm{e}^{-\frac{\nu^\frac{1}{3}}{16}t}  (C^{-}_{in,s})^2 .
		\end{align}
	\end{lemma}
	
	Combining the above Lemma and (\ref{energy approx}), we easily get the proof of Proposition \ref{proposition for th 1}.
	\begin{proof}
		(Proof of Lemma \ref{lemma 1}) To apply Gr\"onwall's inequality for the proof of (\ref{bound 1 in lemma 1}) and (\ref{bound 2 in lemma 1}), we begin by computing the time derivative of $E_\delta(t)$.
		From the definition of $Z_1(t),Z_2(t),Z_3(t)$, we have
		\begin{align}
			\label{derivative of Z_1}
			\partial_{t}Z_1=&-\frac{\partial_{t}m}{m} Z_1-\frac{1}{4}\frac{\partial_{t}\alpha}{\alpha}Z_1-\frac{1}{M}\alpha^\frac{1}{2}Z_2 ,\\
			\label{derivative of Z_2}
			\partial_{t}Z_2=&-\left(\frac{\partial_{t}m}{m}+\mu \alpha\right) Z_2+ \frac{1}{4}\frac{\partial_{t}\alpha}{\alpha}Z_2+\left(\frac{1}{M}\alpha^\frac{1}{2}+2M\frac{k^2}{\alpha^\frac{3}{2}}+M\frac{4\pi \alpha^\frac{1}{2}}{\alpha+4\pi M^2 \delta}\right) Z_1  \nonumber \\
			&-\frac{2k^2}{\alpha}Z_3-2\nu M^2 \frac{k^2}{\alpha}Z_2 ,\\
			\label{derivative of Z_3}
			\partial_{t}Z_3=&-\left(  \frac{\partial_{t}m}{m}+\nu \alpha  \right)Z_3 -\frac{3}{4}\frac{\partial_{t}\alpha}{\alpha}Z_3+\nu (\mu-\nu)M^2\alpha Z_2-\nu M^2\frac{\partial_{t}\alpha}{\alpha}Z_2 \nonumber \\
			&-\nu M^3 \left(  \frac{2k^2}{\alpha^\frac{3}{2}} +\frac{4\pi \alpha^\frac{1}{2}}{\alpha+4\pi M^2 \delta} \right)Z_1 +\nu M^2 \frac{2k^2}{\alpha}Z_3+2\nu^2M^4\frac{k^2}{\alpha}Z_2 .
		\end{align}
		From \eqref{derivative of Z_1}\eqref{derivative of Z_2}\eqref{derivative of Z_3}, direct computation yields
		\begin{align}
			\label{derivative of 1st term in energy}
			\frac{1}{2}\frac{\mathrm{d}}{\mathrm{dt}}\vert Z_1 \vert ^2=-\frac{\partial_{t}m}{m}\vert Z_1 \vert ^2 -\frac{1}{4}\frac{\partial_{t}\alpha}{\alpha}\vert Z_1 \vert ^2-\frac{\alpha^\frac{1}{2}}{M}\mathrm{Re}(\bar{Z}_1Z_2) ,
		\end{align}
		\begin{align}
			\label{derivative of 4th term in energy}
			\frac{1}{2}\frac{\mathrm{d}}{\mathrm{dt}}\vert Z_2 \vert ^2 =&-\left(  \frac{\partial_{t}m}{m}+\mu \alpha \right)\vert Z_2 \vert ^2+\frac{1}{4}\frac{\partial_{t}\alpha}{\alpha}\vert Z_2 \vert ^2+\frac{1}{M}\alpha^\frac{1}{2}\mathrm{Re}(Z_1\bar{Z}_2) 
			\nonumber \\
			&+2M\frac{k^2}{\alpha^\frac{3}{2}}\mathrm{Re}(Z_1\bar{Z}_2)
			  +M\alpha^\frac{1}{2}\frac{4\pi}{\alpha+4\pi M^2 \delta}\mathrm{Re}(Z_1\bar{Z}_2)\nonumber\\
			 &-\frac{2k^2}{\alpha}\mathrm{Re}(Z_3\bar{Z}_2)-2\nu M^2 \frac{k^2}{\alpha} \vert Z_2 \vert ^2 ,
		\end{align}
	and
		\begin{align}
			\label{derivative of 5th term in energy}
			\frac{1}{2} \frac{\mathrm{d}}{\mathrm{dt}} \vert Z_3 \vert ^2 =&- \left( \frac{\partial_{t}m}{m}+\nu \alpha  \right) \vert Z_3 \vert ^2-\frac{3}{4}\frac{\partial_{t}\alpha}{\alpha}\vert Z_3 \vert ^2+\nu(\mu-\nu)M^2\alpha \mathrm{Re}(Z_2\bar{Z}_3)
			\nonumber\\
			&-\nu M^2\frac{\partial_{t}\alpha}{\alpha} \mathrm{Re}(Z_2\bar{Z}_3)
				-\nu M^3 \left( \frac{2k^2}{\alpha^\frac{3}{2}}+\alpha^\frac{1}{2}\frac{4\pi}{\alpha+4\pi M^2\delta} \right)\mathrm{Re}(Z_1\bar{Z}_3) \nonumber \\
			&+\nu M^2\frac{2k^2}{\alpha}\vert Z_3 \vert ^2+2\nu ^2 M^4\frac{k^2}{\alpha}\mathrm{Re}(Z_2\bar{Z}_3) .
		\end{align}
	The term with coefficient $-\frac{\alpha^\frac{1}{2}}{M}$ in (\ref{derivative of 1st term in energy}) can be canceled by (\ref{derivative of 4th term in energy}). The last third term in (\ref{derivative of 4th term in energy}) generates from electric field. And the corresponding multiplier $M\alpha^\frac{1}{2}\dfrac{4\pi}{\alpha+4\pi M^2 \delta}$ is not integrable on $\mathbb{R}$ with respect to $t$. Thus, we encounter considerable challenges when using Gr\"onwall's inequality. By means of careful construction, we introduce the term $M^2\dfrac{4\pi}{\alpha+4\pi M^2 \delta} \vert Z_1 \vert ^2$ in the energy functional and the time derivative
		\begin{align}
			\label{derivative of 3rd term in energy}
			\frac{\mathrm{d}}{\mathrm{dt}} \left(  \frac{1}{2}M^2\frac{4\pi}{\alpha+4\pi M^2 \delta} \vert Z_1 \vert ^2 \right)
			=&-\frac{1}{2}M^2\partial_{t}\alpha \frac{4\pi}{(\alpha+4\pi M^2 \delta)^2}\vert Z_1 \vert ^2-M^2\frac{\partial_{t}m}{m}\frac{4\pi}{\alpha+4\pi M^2 \delta} \vert Z_1 \vert ^2 \nonumber \\
			&-\frac{1}{4}M^2\frac{\partial_{t}\alpha}{\alpha}\frac{4\pi }{\alpha+4\pi M^2 \delta}\vert Z_1 \vert ^2-M\alpha^\frac{1}{2}\frac{4\pi}{\alpha+4\pi M^2 \delta}\mathrm{Re}(\bar{Z}_1Z_2)  .
		\end{align}
	generates a term that is opposite in sign to the last third term in (\ref{derivative of 4th term in energy}).
	
	There are terms appearing in (\ref{derivative of 1st term in energy}) and (\ref{derivative of 4th term in energy}) with undesirable coefficient $\frac{\partial_{t}\alpha}{\alpha}$, which is also not integrable on $\mathbb{R}$ with respect to $t$, significantly impacting the application of Gr\"onwall's inequality. The reason we introduced the first mixed term in $E_\delta(t)$ is that its time derivative can eliminate these undesirable terms. The time derivative of the first mixed term is
		\begin{align}
			\label{derivative of 6th term in energy}
			\frac{\mathrm{d}}{\mathrm{dt}} \left( \frac{M}{4} \frac{\partial_{t}\alpha}{\alpha ^\frac{3}{2}} \mathrm{Re}(\bar{Z}_1Z_2) \right)
			=&\frac{M}{4}\left(  \frac{2k^2}{\alpha^\frac{3}{2}}-\frac{3}{2}\frac{(\partial_{t}\alpha)^2}{\alpha^\frac{5}{2}}  \right)\mathrm{Re}(\bar{Z}_1Z_2) -\frac{M}{2}\frac{\partial_{t}m}{m}\frac{\partial_{t}\alpha}{\alpha^\frac{3}{2}}\mathrm{Re}(\bar{Z}_1Z_2)\nonumber \\
			&-\mu \frac{M}{4} \frac{\partial_{t}\alpha}{\alpha^\frac{1}{2}}\mathrm{Re}(\bar{Z}_1Z_2)-\nu M^3 \frac{k^2\partial_{t}\alpha}{2\alpha^\frac{5}{2}}\mathrm{Re}(\bar{Z}_1Z_2)\nonumber \\
			&-M\frac{k^2\partial_{t}\alpha}{2\alpha^\frac{5}{2}}\mathrm{Re}(\bar{Z}_1Z_3) 
			-\frac{1}{4}\frac{\partial_{t}\alpha}{\alpha}\left(  \vert Z_2 \vert ^2 -  \vert Z_1 \vert ^2 \right) \nonumber \\
			&+ \left( M^2 \frac{k^2\partial_{t}\alpha}{2\alpha^3} +M^2 \frac{\pi}{\alpha+4\pi M^2\delta} \frac{\partial_{t}\alpha}{\alpha} \right) \vert Z_1 \vert ^2 .
		\end{align}
	To obtain dissipation for $Z_1$, we introduce the second mixed term in the energy functional. The time derivative of second mixed term is
		\begin{align}
			\label{derivative of 7th term in energy}
			\frac{\mathrm{d}}{\mathrm{dt}} \left(  -\gamma\alpha^{-\frac{1}{2} } \mathrm{Re}(\bar{Z}_1Z_2)  \right)=&\frac{\gamma}{2}\frac{\partial_{t}\alpha}{\alpha^\frac{3}{2}}\mathrm{Re}(\bar{Z}_1Z_2) +2\gamma \frac{\partial_{t}m}{m} \alpha^{-\frac{1}{2}} \mathrm{Re}(\bar{Z}_1Z_2) +\gamma \mu \alpha^{\frac{1}{2} } \mathrm{Re}(\bar{Z}_1Z_2) \nonumber \\
			&+2\gamma \nu M^2\frac{k^2}{\alpha^\frac{3}{2}}\mathrm{Re}(\bar{Z}_1Z_2)+2\gamma \frac{k^2}{\alpha^\frac{3}{2}}\mathrm{Re}(\bar{Z}_1Z_3) +\frac{1}{M} \gamma  \vert Z_2 \vert ^2   \nonumber \\
			&-\frac{\gamma}{M} \left(  1+2M^2\frac{k^2}{\alpha^2} +M^2 \frac{4 \pi}{\alpha+4\pi M^2\delta} \right) \vert Z_1 \vert ^2 .
		\end{align}		
	The term $M^2 \dfrac{(\partial_{t}\alpha)^2}{\alpha^3} \vert Z_1 \vert ^2 $ in $E_\delta(t)$ guarantees the coercive property of $E_\delta(t)$ after the introduction of two mixed terms. Directly computing, we get its time derivative as follow
		\begin{align}
			\label{derivative of 2nd term in energy}
			\frac{M^2}{2} \frac{\mathrm{d}}{\mathrm{dt}} \left( \frac{(\partial_{t}\alpha)^2}{\alpha^3} \vert Z_1 \vert ^2 \right)
			=& M^2 \left(   \frac{2k^2\partial_{t}\alpha}{\alpha^3}-\frac{3}{2}\frac{(\partial_{t} \alpha)^3}{\alpha^4}   \right)  \vert Z_1 \vert ^2- M^2 \frac{\partial_{t} m}{m}
			\frac{(\partial_{t}\alpha)^2}{\alpha^3} \vert Z_1 \vert ^2\nonumber \\
			&-\frac{M^2}{4}\frac{(\partial_{t}\alpha)^3}{\alpha^4}\vert Z_1 \vert ^2- M\frac{(\partial_{t}\alpha)^2}{\alpha^\frac{5}{2}} \mathrm{Re}(\bar{Z}_1Z_2) .
		\end{align}

		Hence, rearranging the terms appearing in (\ref{derivative of 1st term in energy})-(\ref{derivative of 2nd term in energy}), we get the time derivative of the energy functional $E_\delta (t)$ as follows
		\begin{align}
			\label{derivative of E(t)}
			\frac{\mathrm{d}}{\mathrm{dt}}E_\delta(t)=&-\left( \frac{\partial_{t}m}{m}+\mu\alpha \right) \vert Z_2 \vert ^2
			-\left(\frac{\partial_{t}m}{m}+\frac{\gamma}{M}\left(1+2M^2\frac{k^2}{\alpha^2}+M^2\frac{4 \pi}{\alpha+4\pi M^2\delta} \right)\right)  \vert Z_1 \vert ^2 \nonumber \\
			&-\left( \frac{\partial_{t}m}{m}+\nu \alpha \right) \vert Z_3 \vert ^2  +\sum_{i=1}^{6}\mathcal{D}_i+\sum_{i=1}^{6}\mathcal{I}_i ,
		\end{align}
		where the diffusion error terms are
		\begin{align}
			&\mathcal{D}_1=\frac{\gamma}{M}\vert Z_2 \vert ^2,
			\quad
			\mathcal{D}_2=\gamma\mu\alpha^\frac{1}{2}\mathrm{Re}(\bar{Z}_1Z_2),\nonumber \\
			&\mathcal{D}_3=\nu (\mu-\nu)M^2\alpha \mathrm{Re}(Z_2\bar{Z}_3),
			\quad
			\mathcal{D}_4=-\nu M^2\frac{\partial_{t}\alpha}{\alpha}\mathrm{Re}(Z_2\bar{Z}_3),\nonumber \\
			&\mathcal{D}_5=-\frac{\mu M}{4}\frac{\partial_{t}\alpha}{\alpha^\frac{1}{2}}\mathrm{Re}(\bar{Z}_1Z_2),
			\quad
			\mathcal{D}_6=-\nu M^3\alpha^\frac{1}{2}\frac{4\pi}{\alpha+4\pi M^2\delta}\mathrm{Re}(\bar{Z}_1Z_3),\nonumber
		\end{align}
		which will be absorbed by the negative terms in (\ref{derivative of E(t)}).
		The integrable error terms are
		\begin{align}
			%%%%%%%
			\mathcal{I}_1=&\bigg(  M^2\left( \frac{5k^2\partial_{t}\alpha}{2\alpha^3}-\frac{7}{4}\frac{(\partial_{t}\alpha)^3}{\alpha^4} \right) +M^2\frac{\pi}{\alpha+4\pi M^2\delta}\frac{\partial_{t}\alpha}{\alpha}-M^2\frac{\partial_{t}\alpha2\pi}{(\alpha+4\pi M^2\delta)^2}\nonumber\\
			&-M^2\frac{4\pi}{\alpha+4\pi M^2\delta}
			\frac{\partial_{t}m}{m}-\frac{1}{4}M^2 \frac{4\pi}{\alpha+4\pi M^2\delta}\frac{\partial_{t}\alpha}{\alpha}\bigg) \vert Z_1 \vert ^2 , \nonumber \\
			%%%%%%%
			\mathcal{I}_2=&M\bigg( \frac{1}{2}\left( \frac{\gamma}{M}-\frac{\partial_{t}m}{m}\right) \frac{\partial_{t}\alpha}{\alpha^\frac{3}{2}}+\left(\frac{5}{2}+2\gamma\nu M\right)\frac{k^2}{\alpha^\frac{3}{2}}-\frac{11}{8}\frac{(\partial_{t}\alpha)^2}{\alpha^\frac{5}{2}}\bigg) \mathrm{Re}(\bar{Z}_1Z_2) \nonumber \\
			&-\nu \frac{k^2M^3\partial_{t}\alpha}{2\alpha^\frac{5}{2}}\mathrm{Re}(\bar{Z}_1Z_2)+2\gamma\frac{\partial_{t}m}{m}\alpha^{-\frac{1}{2}}\mathrm{Re}(\bar{Z}_1Z_2) ,\nonumber \\
			%%%%%%%%
			\mathcal{I}_3=&-\frac{3}{4}\frac{\partial_{t}\alpha}{\alpha}\vert Z_3 \vert ^2+2\nu M^2 \frac{k^2}{\alpha}\vert Z_3 \vert ^2 ,\nonumber \\
			%%%%%%%%
			\mathcal{I}_4=&M\bigg( 2\left(\frac{\gamma}{M}-\nu M^2\right)\frac{k^2}{\alpha^\frac{3}{2}}-\frac{k^2\partial_{t}\alpha}{2\alpha^\frac{5}{2}} \bigg)\mathrm{Re}(\bar{Z}_1Z_3) ,\nonumber \\
			%%%%%%%%%
			\mathcal{I}_5=&\bigg(-2\frac{k^2}{\alpha}+2\nu^2M^4\frac{k^2}{\alpha} \bigg)\mathrm{Re}(\bar{Z}_2Z_3) ,\nonumber \\
			\mathcal{I}_6=&-M^2\frac{\partial_{t}m}{m}\frac{(\partial_{t}\alpha)^2}{\alpha^3}\vert Z_1 \vert ^2-2\nu \frac{k^2M^2}{\alpha}\vert Z_2 \vert ^2 . \nonumber
		\end{align}
		among which the Fourier multipliers are integrable on $\mathbb{R}$ with respect to $t$, except the first multiplier in $\mathcal{I}_3$. In addition, this term is tricky  and leads to loss of derivatives when estimating the energy functional.

		We then continue with providing suitable bounds on the terms $\mathcal{D}_i$. 
		In order to absorb $\mathcal{D}_1$ by the first term in (\ref{derivative of E(t)}), recalling the crucial property (\ref{property about m}), we choose $\gamma=\frac{M\nu^\frac{1}{3}}{4}$ and we have 
		\begin{align}
			\label{bound for D_1}
			\abs{\mathcal{D}_1}=\abs{\frac{\gamma}{M}\abs{Z_2}^2}=\frac{\nu^\frac{1}{3}}{4}\abs{Z_2}^2.
		\end{align}
		Since $M\leq\mathrm{min}\{\mu^{-1},\lambda^{-\frac{1}{2}},\nu^{-\frac{1}{3}}\}$ and $\gamma=\frac{M\nu^\frac{1}{3}}{4}$, we deduce $\gamma \leq \frac{1}{4}, M\mu \leq 1$ and 
		\begin{align}
			\abs{\mathcal{D}_2}
			\leq \frac{\gamma \mu}{2} \alpha \abs{Z_2}^2+\frac{\gamma \mu}{2}\abs{Z_1}^2 
			\leq \frac{\mu \alpha}{8}\abs{Z_2}^2 +\frac{\gamma}{2M}\abs{Z_1}^2 .
		\end{align}
	Due to $M\leq \lambda^{-\frac{1}{2}}$, we bound $\mathcal{D}_3$ as follow  
		\begin{align}
			\abs{\mathcal{D}_3}
			\leq \nu \alpha \abs{Z_2}\abs{Z_3}
			\leq \frac{\nu \alpha}{2} \abs{Z_2}^2+\frac{\nu \alpha}{2} \abs{Z_3}^2.
		\end{align}
		Applying $M\leq\mathrm{min}\{\mu^{-1},\lambda^{-\frac{1}{2}},\nu^{-\frac{1}{3}}\}$ again, we get $\nu M^2 \leq 1$ and
		\begin{align}
			\label{estimate on D_4}
			\abs{\mathcal{D}_4}
			\leq \frac{\nu}{2} \xkh{\frac{1}{2}\abs{Z_3}^2+2M^4 \frac{(\partial_{t}\alpha)^2}{\alpha^2}\abs{Z_2}^2} 
			\leq \frac{\nu}{4} \abs{Z_3}^2 + 4M^2 \frac{k^2}{\alpha}\abs{Z_2}^2 .
		\end{align}
		To control $\mathcal{D}_5$, we have
		\begin{align}
			\label{estimate on D_5}
			\abs{\mathcal{D}_5}
			\leq \frac{\mu}{2} \xkh{ \frac{\alpha}{8} \abs{Z_2}^2+\frac{1}{2}M^2 \frac{(\partial_{t}\alpha)^2}{\alpha^2} \abs{Z_1}^2       } 
			 \leq \frac{\mu \alpha}{16} \abs{Z_2}^2 + \mu M^2 \frac{k^2}{\alpha }\abs{Z_1}^2.
		\end{align} 
		Differing from the case of Navier-Stokes equations, we have an additional term $\mathcal{D}_6$, the presence of which is due to the existence of self-consistent electric field. The coefficient of $\mathcal{D}_6$ contains $M^3$, but the multiplier is not integrable in time, hindering the use of Gr\"onwall's inequality. However, thanks to viscosity, we bound $\mathcal{D}_6$ as follow
		\begin{align}
			\label{estimate on D_6}
			\abs{\mathcal{D}_6} 
			\leq \frac{\nu}{2} \abs{Z_3} \frac{M^3 8 \pi}{\alpha^\frac{1}{2}}\abs{Z_1} 
			\leq \frac{\nu^2}{8} \abs{Z_3}^2 + \frac{M^6(8\pi)^2}{2\alpha} \abs{Z_1}^2 .
		\end{align}
		It should be noted that the last terms in (\ref{estimate on D_4})-(\ref{estimate on D_6}) can not be absorbed by the negative terms in (\ref{derivative of E(t)}), but their integrability property in time is sufficient for needs of Gr\"onwall's lemma.

		For the purpose of using Gr\"onwall's inequality, we now proceed to bound $\mathcal{I}_1,\mathcal{I}_2, \cdots ,\mathcal{I}_6$ using Fourier multipliers with time integrability on $\mathbb{R}$. Recalling $\frac{\partial_{t}m}{m}\geq 0$, we have
		\begin{align}
			\mathcal{I}_1
			\leq& \xkh{M^2\xkh{\frac{5k^2}{2\alpha}+\frac{7}{4}\frac{k^2}{\alpha}} +M^2 \frac{\pi k^2}{\alpha} +2\pi M^2\frac{k^2}{\alpha} +\frac{1}{4} M^2 4\pi \frac{k^2}{\alpha} } \abs{Z_1}^2	\nonumber \\
			\leq &\: C_1 M^2 \frac{k^2}{\alpha } \abs{Z_1}^2.
		\end{align}
		Since $\gamma\leq\frac{1}{4}$, $M\leq\mathrm{min}\{\mu^{-1},\lambda^{-\frac{1}{2}},\nu^{-\frac{1}{3}}\}$ and $\abs{\partial_{t}\alpha}\leq 2 \abs{k} \alpha^\frac{1}{2}$, we get
			\begin{align}
				\abs{ \mathcal{I}_2 }
				\leq& \xkh{C_2(1+M)\frac{k^2}{\alpha}+\xkh{\frac{M}{2}+\frac{1}{2}} \frac{\partial_{t}m}{m}  } \xkh{ \abs{Z_1}^2 +\abs{Z_2}^2},\\
				%%%%%%%
				\abs{\mathcal{I}_4}
				\leq& C_4(1+M)\frac{k^2}{\alpha} \xkh{\abs{Z_1}^2+\abs{Z_3}^2} ,\\
				%%%%%%%
				\abs{\mathcal{I}_5} 
				\leq& C_5 (1+M^2) \frac{k^2}{\alpha} \xkh{ \abs{Z_2}^2+\abs{Z_3}^2 } .
			\end{align}
		Unlike $\mathcal{I}_1,\cdots, \mathcal{I}_5$, the sign of $\mathcal{I}_6$ is clear. We bound $\mathcal{I}_6$ as follow
			\begin{align}
					\mathcal{I}_6\leq& 0 .
			\end{align}
		As for $\mathcal{I}_3$, since the multiplier of the first term is not integrable in time on $\mathbb{R}$, we proceed to make a cut-off as follows.
			\begin{align}
				\mathcal{I}_3=&-\frac{3}{4}\frac{\partial_{t}\alpha}{\alpha}\vert Z_3 \vert ^2+2\nu M^2 \frac{k^2}{\alpha}\vert Z_3 \vert ^2 \nonumber \\
				\leq& -\frac{3}{4}\frac{\partial_{t}\alpha}{\alpha} \chi_{t<\frac{\xi}{k}} \abs{Z_3}^2 +\frac{2k^2}{\alpha}\abs{Z_3}^2 .			
			\end{align}
		As a result of the cut-off on $\frac{\partial_{t}\alpha}{\alpha}$, the loss of derivatives will occur later on.
		
		Combining the above estimates on $\mathcal{D}_i$ and $\mathcal{I}_i$, we infer
		\begin{align}
			\label{estimate 1 of derivative of E(t)}
			\frac{\mathrm{d}}{\mathrm{dt}}E_\delta(t)
			\leq  
			%%%%%%
			\Bigg[&
				-\frac{\nu ^\frac{1}{3}}{4} \xkh{1+2M^2\frac{k^2}{\alpha^2}+M^2\frac{4\pi}{\alpha+4\pi M^2\delta}}+\frac{1}{8} \nu^\frac{1}{3} 
			\Bigg]   \abs{Z_1}^2 \nonumber \\
			+\Bigg[&
			-\frac{\partial_{t} m}{m}-\mu \alpha+ \frac{\nu^\frac{1}{3}}{4}+\frac{\mu \alpha}{8} +\frac{\nu \alpha }{2} 
			+ \frac{\mu \alpha}{16}
			\Bigg] \abs{Z_2}^2 \nonumber \\
			+\Bigg[&
			-\frac{\partial_{t} m}{m} -\nu \alpha +\frac{\nu \alpha}{2} +\frac{\nu}{4} +\frac{\nu^2}{8} 
			\Bigg] \abs{Z_3}^2 \nonumber \\
			+\Bigg(&   
			-\frac{\partial_{t} m}{m} + \mu M^2 \frac{k^2}{\alpha} 
			+\frac{M^6(8\pi)^2k^2}{2\alpha} 
			+C_1M^2\frac{k^2}{\alpha} \nonumber \\
			&+C_2(1+M)\frac{k^2}{\alpha} +\xkh{\frac{M}{2}+\frac{1}{2}} \frac{\partial_{t} m}{m}
			 +C_4(1+M)\frac{k^2}{\alpha}
			\Bigg) \abs{Z_1}^2 \nonumber \\
			%%%%%%%
			+ 
			\Bigg( &
				4M^2\frac{k^2}{\alpha}
				+C_2(1+M)\frac{k^2}{\alpha} +\xkh{\frac{M}{2}+\frac{1}{2}}\frac{\partial_{t} m}{m}  
				+C_5(1+M^2)\frac{k^2}{\alpha}
			\Bigg) \abs{Z_2}^2 \nonumber \\
			%%%%%%%
			+ \Bigg( &
				\frac{2k^2}{\alpha} 
				+C_4(1+M)\frac{k^2}{\alpha}+C_5(1+M^2)\frac{k^2}{\alpha}
			\Bigg) \abs{Z_3}^2 \nonumber \\
			%%%%%%
			-\frac{3}{4}& \frac{\partial_{t} \alpha}{\alpha}  \chi_{t<\frac{\xi}{k}} \abs{Z_3}^2 .
		\end{align}
		Recalling the crucial property $\frac{\partial_{t}m}{m}+\nu \alpha \geq \nu^\frac{1}{3}$, three brackets $[ \:\:\:]$ on the right-hand side of (\ref{estimate 1 of derivative of E(t)})  can be bounded as follow
		\begin{align}
			&[ \:\:\:]_1\leq -\frac{\nu^\frac{1}{3}}{16} \xkh{ 1+4M^2\frac{k^2}{\alpha^2}+M^2\frac{4\pi}{\alpha+4\pi M^2\delta}} ,\nonumber\\
			%%%%%%%%%%%%%%%%%%%%%%%%%%%%%%%%%%%%%%%%%%
			&[ \:\:\:]_2
			\leq -\frac{\nu^\frac{1}{3}}{16} ,\nonumber \\
			%%%%%%%%%%%%%%%%%%%%%%%%%%%%%%%%%%%%
			&[ \:\:\:]_3  \leq -\frac{1}{16}\nu^\frac{1}{3}.\nonumber
		\end{align}
		Combining the above estimates, we have
		\begin{align}
			\frac{\mathrm{d}}{\mathrm{dt}}E_\delta(t)\leq &
			-\frac{\nu^\frac{1}{3}}{16}
			\Bigg(
			\xkh{1+4M^2\frac{k^2}{\alpha^2}+M^2\frac{4\pi}{\alpha+4\pi M^2\delta}} \abs{Z_1}^2+\abs{Z_2}^2+\abs{Z_3}^2 
			\Bigg) \nonumber \\
			&+ \Bigg(
			C_M\frac{k^2}{\alpha}+ \xkh{\frac{M}{2}+\frac{1}{2}} \frac{\partial_{t} m}{m}
			\Bigg) \xkh{\abs{Z_1}^2+\abs{Z_2}^2+\abs{Z_3}^2 } \nonumber \\
			&-\frac{3}{4} \frac{\partial_{t} \alpha}{\alpha}  \chi_{t<\frac{\xi}{k}} \abs{Z_3}^2 ,
		\end{align}
		where $C_M=C(1+M^6)$, and $C$ is a suitable constant , not depending on $M$.

		We claim that $\forall t$, $\abs{Z_3(t)}^2\leq 2E_\delta(t)$. Otherwise, there exists $t_0$, such that $\abs{Z_3(t_0)}^2 > 2E_\delta(t_0)$. Form (\ref{coercive 2}), we have 
		\begin{align}
			E_\delta(t_0)\geq& \frac{1}{4} \Bigg(   \left( 	1+M^2\frac{(\partial_{t}\alpha)^2}{\alpha^3} + M^2\frac{4\pi}{\alpha+4\pi M^2\delta} \right)  \vert Z_1\vert^2 +\vert Z_2\vert^2 +2\vert Z_3\vert^2  \Bigg) (t_0) \nonumber \\
			>&\frac{1}{4} \Bigg(  \left(1+M^2\frac{(\partial_{t}\alpha)^2}{\alpha^3} + M^2\frac{4\pi}{\alpha+4\pi M^2\delta} \right) \vert Z_1\vert^2+\vert Z_2\vert^2\Bigg) (t_0) +E_\delta(t_0).
		\end{align}
		It follows that $E_\delta(t_0)>E_\delta(t_0)$, which is a contradiction. Thus, $\abs{Z_3(t)}^2\leq 2E_\delta(t)$, for all $t$.
		Hence, we get 
		\begin{align}
			\frac{\mathrm{d}}{\mathrm{dt}}E_\delta(t)\leq -\frac{\nu^\frac{1}{3}}{16} E_\delta(t) 
			+ \Bigg( 
			-\frac{3}{2}\frac{\partial_{t} \alpha}{\alpha}  \chi_{t<\frac{\xi}{k}} 
			+4 C_M\frac{k^2}{\alpha}+4 \xkh{\frac{M}{2}+\frac{1}{2}} \frac{\partial_{t} m}{m}
			\Bigg) E_\delta(t) .
		\end{align}
		Applying Gr\"onwall's inequality, we obtain
		\begin{align}
			\label{estimate of E(t)}
			E_\delta(t)
			&\leq C \:  \mathrm{e}^{-\frac{\nu^\frac{1}{3}}{16}t}
			\frac{\xkh{\alpha(0)}^\frac{3}{2}}{\xkh{\alpha(t)}^\frac{3}{2}\chi_{t<\frac{\xi}{k}} + \xkh{\alpha(\frac{\xi}{k})}^\frac{3}{2}\chi_{t>\frac{\xi}{k}}    } 
			E_\delta(0) \nonumber \\
			&\leq C \:  \mathrm{e}^{-\frac{\nu^\frac{1}{3}}{16}t} \lan{k,\xi}^3 E_\delta(0)   . 
		\end{align}
		Summing $k$ and integrating $\xi$ in (\ref{estimate of E(t)}), we deduce that
		\begin{align}
			\sum_{k\neq0} \int_{\mathbb{R}} E_\delta(t) d\xi \lesssim 
			 &\: \mathrm{e}^{-\frac{\nu^\frac{1}{3}}{16}t} \sum_{k\neq0} \int_{\mathbb{R}} \lan{k,\xi}^3 \xkh{  \abs{Z_1(0)}^2 +\abs{Z_2(0)}^2 +\abs{Z_3(0)}^2 }  d\xi \nonumber \\\
			\lesssim & \: \mathrm{e}^{-\frac{\nu^\frac{1}{3}}{16}t}
			\xkh{
				\frac{1}{M^2}\norm{\Pi^{in}}^2_{H^{s+1}}
				+ \norm{\Psi^{in}}^2_{H^s}
				+ \norm{F^{in}-\nu M^2 \Psi^{in}}^2_{H^s}
			} . \nonumber 
		\end{align}
		Hence, recalling the definition of $C^\pm_{in,s}$, we conclude the proof of Lemma \ref{lemma 1}.
	\end{proof}
	
	\begin{remark}
		\label{remark on loss of derivative}
		The time derivative of $\abs{Z_3}^2$ generates the term $-\frac{3}{4} \frac{\partial_{t} \alpha}{\alpha} \abs{Z_3}^2$. Since $\frac{\partial_{t}\alpha}{\alpha}$ is not integrable on $\mathbb{R}$, we need to make a cut-off when applying Gr\"onwall's inequality and eventually it leads to the appearance of $\lan{k,\xi}^3$, \ie  the loss of derivatives. In addition, $-\frac{\partial_{t} \alpha}{\alpha} \abs{Z_3}^2$ will always exist in the time derivative of $\abs{Z_3}^2$, unless $j=0$, i.e. $Z_3=\lan{k,\xi}^s m^{-1}\alpha^0 \xkh{\widehat{F}-\nu M^2 \widehat{\Psi}} $, but as discussed in Remark \ref{remark on definition of Z_1 Z_2 Z_3}, $j\leq -\frac{3}{4}$ is needed. 
	\end{remark}
	Finally, we present the proof of Corollary \ref{corollary for th 1} .
	\begin{proof}(Proof of Corollary \ref{corollary for th 1})
		To bound $\norm{\widehat{\Gamma}_{\pm}(t)}_{H^s} $, where $\Gamma_\pm=F_\pm-\Pi_\pm$, we firstly derive the unified form equation satisfied by $\Gamma_\pm$ and $F_\pm$ from (\ref{unified form equ 1}) and (\ref{unified form equ 3}), namely
		\begin{align}
			\label{equation of F Pi }
			\partial_{t}\widehat{F}=-\nu \alpha \widehat{F} + \nu \alpha \widehat{\Pi}.
		\end{align}
		 By Duhamel's formular, we have
		\begin{align}
			\widehat{F}(t)=\: \mathrm{e}^{\mathcal{L}_\nu(t)} \widehat{F}^{in} 
			+ \nu \int_{0}^{t} 
			\mathrm{e}^{ \mathcal{L}_\nu(t)-\mathcal{L}_\nu(\tau) } \alpha(\tau)\widehat{\Pi}(\tau) d\tau ,
		\end{align}
		where 
			\begin{align}
				\mathcal{L}_\nu(t,k,\xi):=-\nu \int_{0}^{t} \alpha(\tau,k,\xi) d\tau. \nonumber
			\end{align}
		The decay estimate on $\widehat{\Pi}$  has been established in Proposition \ref{proposition for th 1}. Thus, we now proceed to the decay estimate on $\widehat{F}$, which leads us to study the solution operator $\mathrm{e}^{\mathcal{L}_\nu (t)}$.
		
		Directly computing, we obtain
		\begin{align}
			\mathcal{L}_\nu(t,k,\xi)=-\nu t \xkh{\frac{1}{3}k^2t^2+k^2+\xi^2-\xi k t}. \nonumber
		\end{align}
		Since 
		\begin{align}
			\label{estimate on L_nu}
			\frac{1}{3}k^2t^2+k^2+\xi^2-\xi k t
			\geq \frac{1}{12} k^2t^2, 
		\end{align}	
		we get 
		\begin{align}
			\norm{\mathrm{e}^{\mathcal{L}_\nu(t)}f}_{H^s}
			\leq \norm{ \mathrm{e}^{-\frac{1}{12}\nu t^3} f }_{H^s} 
			\lesssim \mathrm{e}^{-\frac{1}{12}\nu^\frac{1}{3} t}  \norm{f}_{H^s} .\nonumber
		\end{align}
		From the definition of $\mathcal{L}_\nu$ and (\ref{estimate on L_nu}), we bound  $\mathcal{L}_\nu(t)-\mathcal{L}_\nu(\tau)$ as follows
		\begin{align}
			\mathcal{L}_\nu(t,k,\xi)-\mathcal{L}_\nu(\tau,k,\xi)
			= \mathcal{L}_\nu(t-\tau,k,\xi-k\tau) 
			\leq -\frac{1}{12} \nu^\frac{1}{3} (t-\tau) .\nonumber
		\end{align}
		Therefore, we get
		\begin{align}
			\norm{\widehat{F}(t)}_{H^s}
			\lesssim \: \mathrm{e}^{-\frac{1}{12}\nu^\frac{1}{3} t} \norm{\widehat{F}^{in}}_{H^s} +\nu \int_{0}^{t}  \mathrm{e}^{ -\frac{1}{12} \nu^\frac{1}{3} (t-\tau) } \norm{  \alpha(\tau)\widehat{\Pi}(\tau) }_{H^s} d\tau .
			\nonumber
		\end{align}
		To bound the integrand, we have
		\begin{align}
			\norm{  \alpha(\tau)\widehat{\Pi}_\pm(\tau) }_{H^s}
			&\lesssim M\lan{\tau}^\frac{5}{2} \norm{ \xkh{\frac{1}{M}\alpha^{-\frac{1}{4}}\widehat{\Pi}_\pm}(\tau) }_{H^{s+\frac{5}{2}}} \nonumber \\
			&\lesssim M\lan{\tau}^\frac{5}{2}  \mathrm{e}^{-\frac{1}{32}\nu^\frac{1}{3} \tau} C^{\pm}_{in,s+\frac{5}{2}} , \nonumber
		\end{align}
		where we have used $\alpha(\tau,k,\xi)\lesssim \lan{\tau}^2 \lan{k,\xi}^2$ and Proposition \ref{proposition for th 1}. It follows that
		\begin{align}
			\label{bound for F hat}
			\norm{\widehat{F}_{\pm}(t)}_{H^s}
			\lesssim \: \mathrm{e}^{-\frac{1}{12}\nu^\frac{1}{3} t} \norm{\widehat{F}_{\pm}^{in}}_{H^s} +M\lan{t}^\frac{1}{2} \mathrm{e}^{-\frac{1}{64}\nu^\frac{1}{3} t} C^{\pm}_{in,s+\frac{5}{2}}.
		\end{align}
		Finally, combining (\ref{bound for F hat}) and Proposition \ref{proposition for th 1}, we obtain
		\begin{align}
			\norm{\widehat{\Gamma}_{\pm}}_{H^s} 
			\lesssim & \norm{\widehat{F}_{\pm}}_{H^s}  + M\lan{t}^\frac{1}{2} \norm{   \frac{1}{M} \alpha^{-\frac{1}{4}}\widehat{\Pi}_{\pm}   }_{H^{s+\frac{1}{2}}} \nonumber \\
			\lesssim & \: \mathrm{e}^{-\frac{1}{12}\nu^\frac{1}{3} t} \norm{\widehat{F}_{\pm}^{in}}_{H^s} +M\lan{t}^\frac{1}{2} \mathrm{e}^{-\frac{1}{64}\nu^\frac{1}{3} t} C^{\pm}_{in,s+\frac{5}{2}}
			+  M\lan{t}^\frac{1}{2}  \mathrm{e}^{-\frac{1}{32}\nu^\frac{1}{3} t} C^{\pm}_{in,s+\frac{1}{2}} \nonumber \\
			=& \: \mathrm{e}^{-\frac{1}{12}\nu^\frac{1}{3}t} \norm{\omega^{in}_{\pm}+\eta^{in}_{\pm}}_{H^s}
			+ M \lan{t}^\frac{1}{2} \mathrm{e}^{-\frac{1}{64}\nu^\frac{1}{3}t} C^{\pm}_{in,s+\frac{5}{2}}
			+M \lan{t}^\frac{1}{2} \mathrm{e}^{-\frac{1}{32}\nu^\frac{1}{3}t} C^{\pm}_{in,s+\frac{1}{2}} . \nonumber
		\end{align}
		The proof of Corollary \ref{corollary for th 1} is complete.
	\end{proof}
	%%%%%%%%%%%%%%%%%%%%%%%%%%%%%%%%%%%%%%%%%%%%%%%%%%%%%%%%%
	%%%%%%%%%%%%%%%%%%%%%%%%%%%%%%%%%%%%%%%%%%%%%%%%%%%%%%%%%
	%%%%%%%%%%%%%%%%%%%%%%%%%%%%%%%%%%%%%%%%%%%%%%%%%%%%%%%%%
	\subsection{Dissipation without loss of derivatives}
	\ 
	
	When $\nu=0$ or $\lambda=0$, the term $\nu(\mu-\nu)M^2 \alpha \widehat{\Psi}$ does not appear in (\ref{equation satisfied by F- nu M^2 Psi}). From Remark \ref{remark on definition of Z_1 Z_2 Z_3} and Remark \ref{remark on loss of derivative}, we can choose $j=0$ \ie $Z_3=\langle k,\xi \rangle ^s m^{-1} (\widehat{F}-\nu M^2\widehat{\Psi})$ and the term $- \frac{\partial_{t} \alpha}{\alpha} \abs{Z_3}^2$ resulting in loss of derivatives will not exist. This indicates that the loss of derivatives is likely a result of the technique we used, rather than being inherent in such a system. Thus, in this subsection, we aim to obtain estimates without loss of derivatives, under the assumption $\eta^{in}_{\pm,0}=\psi^{in}_{\pm,0}=\omega^{in}_{\pm,0}=0$.

We consider the following model
\begin{align}
\partial_{t} f= \frac{\partial_{t} \alpha}{\alpha} f -\nu \alpha f,
\end{align}
which is a toy model from the structure of (\ref{unified form equ 2}). For $t\leq \frac{\xi}{k}$, $\frac{\partial_{t} \alpha}{\alpha} \leq 0$ and leads to dissipation on $f$. For $t \geq \frac{\xi}{k}$, $\frac{\partial_{t} \alpha}{\alpha} \geq 0$ and leads growth on $f$. We also note that $\forall t$, $-\nu \alpha f$ is always a dissipation term. If $t$ is far away from  critical times, the growth coming from $\frac{\partial_{t} \alpha}{\alpha} $ can be balanced by the dissipation term $-\nu \alpha f$. That is, for any $\beta >0$, we have
\begin{align}
\label{equ 1 dissipation overcome growth}
\nu \alpha (t,k,\xi)\geq \beta^2 \nu^\frac{1}{3} , \quad  \text{if} \: \abs{t-\frac{\xi}{k}} \geq \beta \nu^{-\frac{1}{3}},  \\
\label{equ 2 dissipation overcome growth}
\frac{\partial_{t} \alpha}{\alpha} (t,k,\xi)\leq \frac{2}{ \sqrt{  1+\xkh{\frac{\xi}{k}-t}^2       }   }
\leq 2\beta^{-1}\nu^\frac{1}{3}  ,\quad \text {if} \: \abs{t-\frac{\xi}{k}} \geq \beta \nu^{-\frac{1}{3}},
\end{align}
Thus, for any fixed $\beta>2$, there holds $\frac{\partial_{t} \alpha}{\alpha} \leq \frac{1}{4} \nu \alpha$, if $ \abs{t-\frac{\xi}{k}} \geq \beta \nu^{-\frac{1}{3}}$.

In order to control the growth near the critical time, for a fixed $\beta \geq 2$ to be chosen later, we introduce the Fourier multiplier used in \cite{Antonelli 2021},
\begin{align}
\begin{cases}
	\partial_{t} w(t,k,\xi)=0   &\text{if} \: t\notin [ \frac{\xi}{k}\:,\: \frac{\xi}{k}+\beta \nu^{-\frac{1}{3}} ] , \\
	\partial_{t} w(t,k,\xi)= \xkh{\frac{\partial_{t} \alpha}{\alpha}w }(t,k,\xi)  
	&\text{if} \: t\in [ \frac{\xi}{k}\:,\: \frac{\xi}{k}+\beta \nu^{-\frac{1}{3}} ] ,  \\
	w(0,k,\xi)=1,
\end{cases}
\end{align}
which is explicitly given by 
\begin{align}
\label{def of w}
w(t,k,\xi)=
\begin{cases}
	1  & \text{if} \: \xi k \geq 0 , \:0\leq t \leq \frac{\xi}{k} ,\\
	1  & \text{if} \: \xi k <0 ,\: \abs{\frac{\xi}{k}} \geq \beta \nu ^{-\frac{1}{3}}, \: t\geq 0 ,\\
	\frac{\alpha}{k^2} & \text{if} \: \frac{\xi}{k}\leq t \leq \frac{\xi}{k}+\beta\nu^{-\frac{1}{3}} , \\
	1+\beta^2\nu^{-\frac{1}{3}}  &  \text{other cases}.
\end{cases}
\end{align}
The $\frac{\partial_{t}w}{w}$ can be considered as a cut-off of $\frac{\partial_{t}\alpha}{\alpha}$, preserving the part near the critical time, eliminating the part far away from the critical time, which can be addressed by the inherent dissipation of this system. 
Then we replace $\alpha^{-\frac{3}{4}}$ by $w^{-\frac{3}{4}}$ in the definition of $Z_1, Z_2, Z_3$, namely introducing the following weighted variables.
 	\begin{align}
 		Z^w_1(t)&:=\frac{1}{M} \lan{k,\xi}^s \xkh{m^{-1} w^{-\frac{3}{4}} \alpha^\frac{1}{2} \widehat{\Pi}}(t) , \\
 		Z^w_2(t)&:=\lan{k,\xi}^s \xkh{m^{-1} w^{-\frac{3}{4}} \widehat{\Psi} }(t)  ,\\
 		Z^w_3(t)&:=\lan{k,\xi}^s \xkh{ m^{-1} w^{-\frac{3}{4}} (\widehat{F}-\nu M^2 \widehat{\Psi}) }(t) .
 	\end{align}
 Directly computing the time derivative of $Z^w_1, Z^w_2, Z^w_3$, we have
 \begin{align}
 	\label{derivative of Z^w_1}
 	\partial_{t}Z^w_1
 	=& -
 	\Bigg(
 	\frac{\partial_{t}m}{m} + \frac{3}{4} \xkh{ \frac{\partial_{t}w}{w} - \frac{\partial_{t}\alpha}{\alpha}}
 	\Bigg) Z^w_1
 	-\frac{1}{4} \frac{\partial_{t}\alpha}{\alpha}Z^w_1 
 	-\frac{1}{M} \alpha^\frac{1}{2} Z^w_2 ,\\
 	%%%%%%%%%%%%%%%
 	\label{derivative of Z^w_2}
 	\partial_{t}Z^w_2
 	=& -
 	\Bigg(
 	\frac{\partial_{t}m}{m} +\mu \alpha + \frac{3}{4} \xkh{ \frac{\partial_{t}w}{w} - \frac{\partial_{t}\alpha}{\alpha}} 
 	\Bigg)Z^w_2 
 	+\frac{1}{4} \frac{\partial_{t}\alpha}{\alpha}Z^w_2 \nonumber \\
 	&+ \xkh{
 		\frac{1}{M} \alpha^\frac{1}{2} + \frac{2Mk^2}{\alpha^\frac{3}{2}}
 		+\frac{4\pi M \alpha^\frac{1}{2}}{\alpha+4\pi M^2\delta}
 	} Z^w_1 
 	-\frac{2k^2}{\alpha}Z^w_3 -2\nu M^2\frac{k^2}{\alpha}Z^w_2 ,\\
 	%%%%%%%%%%%%%%%%
 	\label{derivative of Z^w_3}
 	\partial_{t}Z^w_3
 	=&
 	-\xkh{ \frac{\partial_{t}m}{m}+\nu \alpha }Z^w_3 -\frac{3}{4}\frac{\partial_{t}w}{w}Z^w_3
 	+2\nu M^2 \frac{k^2}{\alpha} Z^w_3 \nonumber \\
 	&+\nu (\mu -\nu)M^2 \alpha Z^w_2 -\nu M^2 \frac{\partial_{t}\alpha}{\alpha} Z^w_2 +2 \nu^2M^4 \frac{k^2}{\alpha}Z^w_2 \nonumber \\
 	&-2\nu M^3 \frac{k^2}{\alpha^\frac{3}{2}}Z^w_1 -\nu M^3 \frac{4\pi \alpha^\frac{1}{2}}{\alpha +4\pi M^2\delta}Z^w_1 .
 \end{align}
 Compared (\ref{derivative of Z_3}) with (\ref{derivative of Z^w_3}), 
 the term $-\frac{3}{4}\frac{\partial_{t}\alpha}{\alpha} Z_3$, leading to loss of derivatives, is replaced by $-\frac{3}{4}\frac{\partial_{t}w}{w} Z^w_3$. Thanks to this adjustment, there is no loss of derivatives in the subsequent estimates.
  
 Compared to (\ref{derivative of Z_1})-(\ref{derivative of Z_2}), (\ref{derivative of Z^w_1})-(\ref{derivative of Z^w_2}) have the same structure and leads us to study the lower bound for  $\delta_\beta\xkh{ \frac{\partial_{t}m}{m} +\nu \alpha }  + \frac{\partial_{t}w}{w}- \frac{\partial_{t}\alpha}{\alpha} $, namely Lemma \ref{lemma for th 2}, which is essential for the subsequent analysis. 
%%%%%%%%%
\begin{lemma} \label{lemma for th 2}
Let $w$ and $m$ be the Fourier multiplier defined in (\ref{def of w}) and (\ref{def of m}) respectly. Then, $\forall t\geq 0, \eta \in \mathbb{R}, k \in \mathbb{Z} \backslash \{0\}$, there hold
\begin{align}
	\label{estimate 1 in second lemma }
	1\leq w(t,k,\xi) &\leq 1+\beta^2 \nu^{-\frac{2}{3}}  , \\
	\label{estimate 2 in second lemma}
	\xkh{w\alpha^{-1}}(t,k,\xi)&\leq \frac{1}{k^2} .
\end{align}
In addition, for any fixed $\delta_\beta$, satisfying $\mathrm{max}\{2(\beta(\beta^2-1))^{-1}, 4\beta^{-1}\} <\delta_\beta<1$, one has
\begin{align}
	\Bigg(
	\delta_\beta\xkh{\frac{\partial_{t}m}{m}+\nu \alpha} + 	\frac{\partial_{t}w}{w} -\frac{\partial_{t} \alpha}{\alpha}
	\Bigg) (t,k,\xi)
	\geq \delta_\beta \nu^\frac{1}{3} ,\\
	\Bigg(
	\delta_\beta\xkh{\frac{\partial_{t}m}{m}+\nu ^\frac{1}{3}} + \frac{\partial_{t}w}{w} -\frac{\partial_{t} \alpha}{\alpha}
	\Bigg) (t,k,\xi)
	\geq \frac{\delta_\beta}{2}\nu^\frac{1}{3}.
\end{align}
\end{lemma}
\begin{proof}(Proof of Lemma \ref{lemma for th 2})
From the definition of $w$, we immediately deduce (\ref{estimate 1 in second lemma }) and (\ref{estimate 2 in second lemma}). When $ t\in \left[ \frac{\xi}{k}\:,\: \frac{\xi}{k}+\beta \nu^{-\frac{1}{3}} \right] $, applying property (\ref{property about m}), we have
\begin{align}
	\delta_\beta\xkh{\frac{\partial_{t}m}{m}+\nu \alpha} + 	\frac{\partial_{t}w}{w} -\frac{\partial_{t} \alpha}{\alpha}
	=
	\delta_\beta\xkh{\frac{\partial_{t}m}{m}+\nu \alpha} 
	\geq \delta_\beta \nu^\frac{1}{3} ,\nonumber \\
	%%%%%%%%%%%%%%
	\delta_\beta\xkh{\frac{\partial_{t}m}{m}+\nu ^\frac{1}{3}} + \frac{\partial_{t}w}{w} -\frac{\partial_{t} \alpha}{\alpha}
	=  \delta_\beta\xkh{\frac{\partial_{t}m}{m}+\nu ^\frac{1}{3}}
	\geq \frac{\delta_\beta}{2}\nu^\frac{1}{3}. \nonumber
\end{align}
When $ t \leq \frac{\xi}{k} $, we have $\partial_{t} \alpha \leq 0$. Applying property (\ref{property about m}) again, we have 
\begin{align}
	\delta_\beta\xkh{\frac{\partial_{t}m}{m}+\nu \alpha} + 	\frac{\partial_{t}w}{w} -\frac{\partial_{t} \alpha}{\alpha}
	\geq
	\delta_\beta\xkh{\frac{\partial_{t}m}{m}+\nu \alpha} 
	\geq \delta_\beta \nu^\frac{1}{3} ,\nonumber \\
	%%%%%%%%%%%%%%
	\delta_\beta\xkh{\frac{\partial_{t}m}{m}+\nu ^\frac{1}{3}} + \frac{\partial_{t}w}{w} -\frac{\partial_{t} \alpha}{\alpha}
	\geq
	\delta_\beta\xkh{\frac{\partial_{t}m}{m}+\nu ^\frac{1}{3}}
	\geq \frac{\delta_\beta}{2}\nu^\frac{1}{3}. \nonumber
\end{align}
When $t\geq \frac{\xi}{k}+\beta \nu^{-\frac{1}{3}}  $, by virtue of (\ref{equ 1 dissipation overcome growth}) and (\ref{equ 2 dissipation overcome growth}), we get
\begin{align}
	\delta_\beta \nu \alpha -\frac{\partial_{t} \alpha}{\alpha}
	\geq \nu^\frac{1}{3}\xkh{ \delta_\beta \beta^2 - 2\beta^{-1}}
	\geq \delta_\beta \nu^\frac{1}{3}, \nonumber \\
	%%%%%%%%%%
	\delta_\beta \nu^\frac{1}{3}  -\frac{\partial_{t} \alpha}{\alpha}
	\geq \nu^\frac{1}{3}\xkh{\delta_\beta-2\beta^{-1}} 
	\geq \frac{\delta_\beta}{2} \nu^\frac{1}{3} .\nonumber 
\end{align}
Thus, the proof of Lemma \ref{lemma for th 2} is established.
\end{proof}
\begin{proposition} \label{proposition for th 2}
Let $s\geq 0$, $\mu \leq \frac{1}{2}$, $M>0$ be such that $M\leq \mathrm{min}\{\mu^{-\frac{1}{2}}, \lambda^{-\frac{1}{2}},\nu^{-\frac{1}{3}}\}$. If $\eta^{in}_{\pm} \in H^{s+1}(\mathbb{T}\times \mathbb{R})$ and $\psi^{in}_{\pm}, \omega^{in}_{\pm} \in H^{s}(\mathbb{T}\times \mathbb{R}) $, then
\begin{align}
	&\frac{1}{M}
	\norm{\xkh{ w^{-\frac{3}{4}}\alpha^\frac{1}{2}\widehat{\Pi}_{\pm}}(t) }_{H^s}
	+ \norm{ \xkh{w^{-\frac{3}{4}}\widehat{\Psi}_{\pm} }(t)}_{H^s}
	+ \norm{\xkh{w^{-\frac{3}{4}} ( \widehat{F}_{\pm}-\nu M^2 \widehat{\Psi}_{\pm}  )}(t)}_{H^s} 
	\nonumber \\
	\lesssim &
	\: \mathrm{e}^{-\frac{\nu^\frac{1}{3}}{32}t}
	\xkh{ \frac{1}{M} \norm{\nabla \eta^{in}_{\pm}}_{H^s} +\norm{\psi^{in}_{\pm}}_{H^s} +\norm{ \eta^{in}_{\pm}+\omega^{in}_{\pm} -\nu M^2 \psi^{in}_{\pm} }_{H^s}   } .
\end{align} 
\end{proposition}
%%%%%%%%
\begin{proof} (Proof of Proposition \ref{proposition for th 2})
Define the energy functional as follows
\begin{align}
	E_\delta^w(t)
	=\frac{1}{2}
	\Bigg(
	\xkh{1+M^2\frac{(\partial_{t}\alpha)^2}{\alpha^3}+\frac{M^2 4\pi}{\alpha+4\pi M^2\delta}} \abs{Z^w_1}^2 +\abs{Z^w_2}^2 +\abs{Z^w_3}^2 \nonumber \\
	+\frac{M}{2} \frac{\partial_{t}\alpha}{\alpha^\frac{3}{2}} \mathrm{Re}(\bar{Z}^w_1 Z^w_2) -2 \gamma \alpha^{-\frac{1}{2}} \mathrm{Re}(\bar{Z}^w_1Z^w_2)
	\Bigg) (t) ,
\end{align}
where $\gamma=\frac{M\nu^\frac{1}{3}}{4}$ is a parameter.
$E^w(t)$ is also a coercive energy functional, satisfying 
\begin{align}
	&E_\delta^w(t)\leq \xkh{\bigg( 1+M^2\frac{(\partial_{t}\alpha)^2}{\alpha^3} + 	M^2\frac{4\pi}{\alpha+4\pi M^2\delta}     \bigg) \vert Z^w_1\vert^2 +\vert Z^w_2\vert^2 +\vert Z^w_3\vert^2   }(t) ,\\
	&E_\delta^w(t)\geq \frac{1}{4} \bigg(   \left( 	1+M^2\frac{(\partial_{t}\alpha)^2}{\alpha^3} + M^2\frac{4\pi}{\alpha+4\pi M^2\delta} \right)  \vert Z^w_1\vert^2 +\vert Z^w_2\vert^2 +2\vert Z^w_3\vert^2  \bigg) (t) .
\end{align}
Through a computation procedure analogous to the one in the proof of Lemma \ref{lemma 1}, we obtain
\begin{align}
	\label{derivative of E^w(t)}
	\frac{\mathrm{d}}{\mathrm{dt}} E_\delta^w(t)
	=&-\Bigg(
	\frac{\partial_{t}m}{m}+\mu \alpha +\frac{3}{4}\xkh{\frac{\partial_{t}w}{w}-\frac{\partial_{t}\alpha}{\alpha}}
	\Bigg) \abs{Z^w_2}^2 \nonumber \\
	& -\Bigg(
	\frac{\partial_{t}m}{m}+\frac{3}{4}\xkh{\frac{\partial_{t}w}{w}-\frac{\partial_{t}\alpha}{\alpha}}
	+\frac{\nu^\frac{1}{3}}{4} \xkh{
		1+2M^2\frac{k^2}{\alpha^2} +\frac{M^2 4\pi}{\alpha+4\pi M^2\delta}
	}
	\Bigg)   \abs{Z^w_1}^2 \nonumber \\
	& -\xkh{  \frac{\partial_{t}m}{m}+\nu \alpha   }\abs{Z^w_3}^2 
	+\sum_{i=1}^{6} \mathcal{D}^w_i +\sum_{i=1}^{6}\mathcal{I}^w_i,
\end{align}
 where $\mathcal{D}^w_i$ is identical to $\mathcal{D}_i$ for all $i$ and $\mathcal{I}^w_i$ is identical to $\mathcal{I}_i$ for $i=4,5,6$, except for the replacement of $Z_j$ by $Z^w_j$.
 The definitions of $\mathcal{I}^w_1$ and $\mathcal{I}^w_2$ are 
\begin{align}
	\mathcal{I}^w_1
	=&\Bigg( 
	M^2\xkh{
		\frac{5}{2}\frac{k^2\partial_{t}\alpha}{\alpha^3} -\frac{7}{4}\frac{(\partial_{t}\alpha)^3}{\alpha^4}
	} 
	+ \frac{M^2 \pi}{\alpha+4\pi M^2\delta}\frac{\partial_{t}\alpha}{\alpha}
	-\frac{M^2}{2}\frac{4\pi \partial_{t}\alpha}{(\alpha+4\pi M^2\delta)^2}
	\nonumber \\
	&\:\:\: -\frac{M^2 4\pi}{\alpha+4\pi M^2\delta}\frac{\partial_{t}m}{m}-\frac{1}{4}\ \frac{M^2 4\pi}{\alpha+4\pi M^2\delta}\frac{\partial_{t}\alpha}{\alpha}
	-\frac{3}{4}M^2\xkh{ \frac{\partial_{t}w}{w}-\frac{\partial_{t}\alpha}{\alpha}    }\frac{(\partial_{t}\alpha)^2}{\alpha^3} \nonumber \\
	&-\frac{3}{4}\frac{M^2 4\pi}{\alpha+4\pi M^2\delta}\xkh{ \frac{\partial_{t}w}{w}-\frac{\partial_{t}\alpha}{\alpha}    }
	\Bigg)	\abs{Z^w_1}^2 ,\nonumber \\
	%%%%%%%%%%%%%%%
	\mathcal{I}^w_2
	=&M
	\Bigg(
	\frac{1}{2}\xkh{\frac{\gamma}{M}-\frac{\partial_{t}m}{m}}\frac{\partial_{t}\alpha}{\alpha^\frac{3}{2}} 
	+\xkh{\frac{5}{2}+2\gamma\nu M}\frac{k^2}{\alpha^\frac{3}{2}}
	-\frac{11}{8} \frac{(\partial_{t}\alpha)^2}{\alpha^\frac{5}{2}}
	\Bigg) \mathrm{Re}(\bar{Z}^w_1Z^w_2) \nonumber \\
	&-\nu M^3 \frac{k^2\partial_{t}\alpha}{2\alpha^\frac{5}{2}}\mathrm{Re}(\bar{Z}^w_1Z^w_2)
	+2\gamma \frac{\partial_{t}m}{m}\alpha^{-\frac{1}{2}}\mathrm{Re}(\bar{Z}^w_1Z^w_2) \nonumber \\
	&-\frac{3}{2}\frac{M}{4}\xkh{ \frac{\partial_{t}w}{w}-\frac{\partial_{t}\alpha}{\alpha}    }\frac{\partial_{t}\alpha}{\alpha^\frac{3}{2}}\mathrm{Re}(\bar{Z}^w_1Z^w_2)
	+\frac{3}{2} \gamma \alpha^{-\frac{1}{2}}\xkh{ \frac{\partial_{t}w}{w}-\frac{\partial_{t}\alpha}{\alpha}    }\mathrm{Re}(\bar{Z}^w_1Z^w_2),\nonumber \\
\end{align}
Upon replacing $Z_j$ by $Z^w_j$ in $\mathcal{I}_1, \mathcal{I}_2$, $\mathcal{I}^w_1$,  There are two more terms in $\mathcal{I}^w_1, \mathcal{I}^w_2$, specifically the terms involving multiplier $\frac{\partial_{t}w}{w}-\frac{\partial_{t}\alpha}{\alpha}$. 
The definition of $\mathcal{I}^w_3$ is
\begin{align}
	\mathcal{I}^w_3
	=\xkh{
			-\frac{3}{4} \frac{\partial_{t}w}{w}+2\nu M^2\frac{k^2}{\alpha}
		} \abs{Z^w_3}^2.
\end{align}
The first term in the bracket is $\frac{\partial_{t}w}{w}$ rather than $\frac{\partial_{t}\alpha}{\alpha}$ in $\mathcal{I}_3$, which leads loss of derivatives.

In addition, it is worth mentioning that the terms  involving multiplier $\frac{\partial_{t}w}{w}-\frac{\partial_{t}\alpha}{\alpha}$ in (\ref{derivative of E^w(t)}) arise naturally, corresponding to the  discarded part while $\frac{\partial_{t}w}{w}$ represents the cut-off on $\frac{\partial_{t}\alpha}{\alpha}$.

We bound $\mathcal{D}^w_1, \mathcal{D}^w_2, \cdots, \mathcal{D}^w_6$ and $\mathcal{I}^w_4, \mathcal{I}^w_5, \mathcal{I}^w_6$ by the same method used for $\mathcal{D}_1, \mathcal{D}_2, \cdots, \mathcal{D}_6$ and $\mathcal{I}_4, \mathcal{I}_5, \mathcal{I}_6$. For $\mathcal{I}^w_1, \mathcal{I}^w_2$, we only need to note that $\frac{\partial_{t}w}{w}-\frac{\partial_{t}\alpha}{\alpha} \leq \frac{\abs{\partial_{t}\alpha}}{\alpha}$ and apply the similar method used for $\mathcal{I}_1, \mathcal{I}_2$. Thus, we get the bound for $\mathcal{D}^w_i$ and $\mathcal{I}^w_i$, except $\mathcal{I}^w_3$, identically to those in the proof of Lemma \ref{lemma 1}, after replacing the $Z^w_j$ by $Z_j$. In view of the definition of $w$ and the restriction on Mach number, we bound $\mathcal{I}^w_3$ as follows
	\begin{align}
		\mathcal{I}^w_3
		\leq \frac{2k^2}{\alpha} \abs{Z^w_3}^2 . \nonumber
	\end{align}
As a consequence of these estimates, we arrive at
\begin{align}
	\label{first estimate of derivative of E^w(t)}
	\frac{\mathrm{d}}{\mathrm{dt}} E_\delta^w(t)
	\leq \Bigg[&
				-\frac{\partial_{t}m}{m} 
				-\frac{3}{4}\xkh{\frac{\partial_{t}w}{w}-\frac{\partial_{t}\alpha}{\alpha}}
				\nonumber \\
				&-\frac{\nu^\frac{1}{3}}{4}
							\xkh{
									1+2M^2 \frac{k^2}{\alpha^2}
									+M^2\frac{4\pi}{\alpha+4\pi M^2\delta}
								}
				+\frac{1}{8}\nu^\frac{1}{3} 
			\Bigg] \abs{Z^w_1}^2
			\nonumber \\
		+\Bigg[&
			-\frac{\partial_{t}m}{m} 
			-\mu \alpha
			-\frac{3}{4} \xkh{\frac{\partial_{t}w}{w}-\frac{\partial_{t}\alpha}{\alpha}}
			+\frac{\nu^\frac{1}{3}}{4}
			+\frac{\mu \alpha}{8}
			+\frac{\nu \alpha}{2}
			+\frac{\mu \alpha}{16}
		\Bigg] \abs{Z^w_2}^2 \nonumber \\	
		+\Bigg[&
			-\frac{\partial_{t}m}{m} 
			-\nu\alpha
			+\frac{\nu\alpha}{2}
			+\frac{\nu}{4}
			+\frac{\nu^2}{8} 
		\Bigg] \abs{Z^w_3}^2  \nonumber \\
		+ \Bigg[&
				\mu M^2 \frac{k^2}{\alpha}
				+\frac{M^6(8\pi)^2}{2\alpha}
				+\widetilde{C}_1M^2\frac{k^2}{\alpha} \nonumber \\
				&+\widetilde{C}_2 (1+M)\frac{k^2}{\alpha} 
				+\xkh{\frac{M}{2}+\frac{1}{2}} \frac{\partial_{t}m}{m}
				+\widetilde{C}_4(1+M) \frac{k^2}{\alpha}
			\Bigg]  \abs{Z^w_1}^2 \nonumber \\
	  +\Bigg[&
			4M^2\frac{k^2}{\alpha} 
			+\widetilde{C}_2 (1+M) \frac{k^2}{\alpha}
			+\xkh{\frac{M}{2}+\frac{1}{2}} 
			\frac{\partial_{t}m}{m}
			+\widetilde{C}_5(1+M^2) \frac{k^2}{\alpha}
		\Bigg] \abs{Z^w_2}^2 \nonumber \\
	 + \bigg[&
			\frac{2k^2}{\alpha}
			+\widetilde{C}_4(1+M)\frac{k^2}{\alpha}
			+\widetilde{C}_5(1+M^2)\frac{k^2}{\alpha}
		\bigg]  \abs{Z^w_3}^2 . 
\end{align}
The multipliers in the last three  brackets $[\:\:\:]$ of (\ref{first estimate of derivative of E^w(t)}) are all integrable in time, while the first three brackets $[\:\:\:]$ contribute to the  decay rate.
Precisely, for fixed $\beta>48$, we have $\mathrm{max}\{2(\beta(\beta^2-1))^{-1}, 4\beta^{-1}\} =\frac{4}{\beta}<\frac{1}{12}<1$. Thus, thanks to Lemma \ref{lemma for th 2}, we obtain
\begin{align}
	\label{pre work 1 for first bracket}
	&\frac{1}{12}\xkh{\frac{\partial_{t}m}{m}+\nu \alpha} + 	\frac{\partial_{t}w}{w} -\frac{\partial_{t} \alpha}{\alpha}
	\geq \frac{1}{12} \nu^\frac{1}{3} ,\\
	%%%%%%%%%
	\label{pre work 2 for first bracket}
	&\frac{1}{12}\xkh{\frac{\partial_{t}m}{m}+\nu ^\frac{1}{3}} + \frac{\partial_{t}w}{w} -\frac{\partial_{t} \alpha}{\alpha}
	\geq \frac{1}{24}\nu^\frac{1}{3}.
\end{align}
We now proceed to bound the first bracket $[\:\:\:]$ in (\ref{first estimate of derivative of E^w(t)}), using (\ref{pre work 1 for first bracket})
\begin{align}
	[\:\:\:]_1
	\leq&   -\frac{3}{4}
			\Bigg(
			\frac{1}{12}
			\xkh{
				\frac{\partial_{t}m}{m}
				+\nu^\frac{1}{3}
			}
			+\xkh{\frac{\partial_{t}w}{w}-\frac{\partial_{t}\alpha}{\alpha}} 
			\Bigg)
			-\frac{\nu^\frac{1}{3}}{16}
			\xkh{
				1+4M^2 \frac{k^2}{\alpha^2}								+M^2\frac{4\pi}{\alpha+4\pi M^2\delta}
			} \nonumber \\
			\leq& -\frac{\nu^\frac{1}{3}}{16}
			\xkh{
				1+4M^2 \frac{k^2}{\alpha^2}								+M^2\frac{4\pi}{\alpha+4\pi M^2\delta}
			} .
\end{align}
Analogously, by (\ref{pre work 2 for first bracket}), the second bracket in (\ref{first estimate of derivative of E^w(t)}) is bounded as
	\begin{align}
		[\:\:\:]_2\leq -\frac{1}{16} \nu^\frac{1}{3} .
	\end{align}
The bound for the third bracket follows directly from the crucial property (\ref{property about m})
	\begin{align}
		[\:\:\:]_3\leq -\frac{1}{16} \nu^\frac{1}{3} . \nonumber
	\end{align}
 Consequently, we obtain
\begin{align}
	\frac{\mathrm{d}}{\mathrm{dt}} E_\delta^w(t)
	\leq &  -\frac{1}{16} \nu^\frac{1}{3}
	\Bigg(
	\xkh{
		1+4M^2 \frac{k^2}{\alpha^2}							+M^2\frac{4\pi}{\alpha+4\pi M^2\delta}
	} \abs{Z^w_1}^2
	+\abs{Z^w_2}^2
	+\abs{Z^w_3}^2
	\Bigg) \nonumber \\
	& + \Bigg(
	\widetilde{C}_M(1+M^6) \frac{k^2}{\alpha}
	+\xkh{\frac{M}{2}+\frac{1}{2}} \frac{\partial_{t}m}{m}
	\Bigg) 
	\xkh{ \abs{Z^w_1}^2
		+\abs{Z^w_2}^2
		+\abs{Z^w_3}^2 } \nonumber \\
	\leq &   -\frac{1}{16} \nu^\frac{1}{3} E_\delta^w(t)
	+ \Bigg(
			4 \widetilde{C}_M(1+M^6) \frac{k^2}{\alpha}
			+4 \xkh{\frac{M}{2}+\frac{1}{2}} \frac{\partial_{t}m}{m}
   	  \Bigg) E_\delta^w(t) . \nonumber
\end{align}
Then applying Gr\"onwall's Lemma, we get
\begin{align}
	E_\delta^w(t) 
	\leq \widetilde{C}\: \mathrm{e}^{-\frac{1}{16} \nu^\frac{1}{3} t} \: E_\delta^w(0) . \nonumber
\end{align}
In view of the definition of $Z^w_1, Z^w_2, Z^w_3$ and the coercivity property of $E_\delta^w(t)$, we also have
\begin{align}
	\sum_{k\neq 0} \int_{\mathbb{R}} E_\delta^w(t) d\xi
	\approx&	\sum_{k\neq 0} \int_{\mathbb{R}} 
	\xkh{
		\abs{Z^w_1}^2
		+\abs{Z^w_2}^2
		+\abs{Z^w_3}^2
	} (t)
	d \xi \nonumber \\
	\approx & \frac{1}{M^2}
	\norm{\xkh{ w^{-\frac{3}{4}}\alpha^\frac{1}{2}\widehat{\Pi}}(t) }^2_{H^s}
	+ \norm{ \xkh{w^{-\frac{3}{4}}\widehat{\Psi} }(t)}^2_{H^s} \nonumber \\
	&+ \norm{\xkh{w^{-\frac{3}{4}} ( \widehat{F}-\nu M^2 \widehat{\Psi}  )}(t)}^2_{H^s}  , \nonumber
\end{align}
where in the last line we have use the assumption  $\eta^{in}_{\pm,0}=\psi^{in}_{\pm,0}=\omega^{in}_{\pm,0}=0$. Therefore, we obtain
\begin{align}
	&\frac{1}{M^2}
	\norm{\xkh{ w^{-\frac{3}{4}}\alpha^\frac{1}{2}\widehat{\Pi}}(t) }^2_{H^s}
	+ \norm{ \xkh{w^{-\frac{3}{4}}\widehat{\Psi} }(t)}^2_{H^s} 
	+ \norm{\xkh{w^{-\frac{3}{4}} ( \widehat{F}-\nu M^2 \widehat{\Psi}  )}(t)}^2_{H^s} \nonumber \\
	\lesssim&\: \mathrm{e}^{-\frac{1}{16} \nu^\frac{1}{3} t}
	\sum_{k\neq 0} \int_{\mathbb{R}}  
	E_\delta^w(0)   d\xi \nonumber \\
	\lesssim& \: \mathrm{e}^{-\frac{1}{16} \nu^\frac{1}{3} t}
	\xkh{ \frac{1}{M^2} \norm{\nabla \Pi^{in}}^2_{H^s} +\norm{\Psi^{in}}^2_{H^s} +\norm{ F^{in} -\nu M^2 \Psi^{in} }^2_{H^s}   } . \nonumber
\end{align}
Hence, conclude the proof of Proposition \ref{proposition for th 2} .
\end{proof}

Appealing to Proposition \ref{proposition for th 2}, we present the proof of Theorem \ref{th 2} as follows.
\begin{proof} (Proof of Theorem \ref{th 2})
	Firstly, Plancherel's Theorem gives 
\begin{align}
	&\norm{\psi_{\pm}}_{L^2} + \frac{1}{M} \norm{\nabla \eta_{\pm}}_{L^2} + \norm{\omega_{\pm}}_{L^2} \nonumber \\
	=& \norm{\widehat{\Psi}_{\pm}}_{L^2} +\frac{1}{M} \norm{\alpha^\frac{1}{2} \widehat{\Pi}_{\pm}}_{L^2}
	+ \norm{\widehat{\Gamma}_{\pm}}_{L^2} . \nonumber
\end{align}
Thanks to 
\begin{align}
	\abs{\widehat{\Gamma}_{\pm}} 
	\leq 
	\abs{\widehat{F}_{\pm}-\nu M^2 \widehat{\Psi}_{\pm}}
	+ M \frac{1}{M} \alpha^\frac{1}{2} \abs{\widehat{\Pi}_{\pm}}
	+ \nu M^2 \abs{\widehat{\Psi}_{\pm}} , \nonumber
\end{align}
and $\nu M^2 \leq 1$, 
we deduce that 
\begin{align}
	&\norm{\psi_{\pm}}_{L^2} + \frac{1}{M} \norm{\nabla \eta_{\pm}}_{L^2} + \norm{\omega_{\pm}}_{L^2}  \nonumber \\
	\lesssim & \norm{ \widehat{\Psi}_{\pm} }_{L^2}
	+ \frac{1}{M}  \norm{ \alpha^\frac{1}{2} \widehat{\Pi}_{\pm} }_{L^2}
	+ \norm{ \widehat{F}_{\pm}-\nu M^2 \widehat{\Psi}_{\pm}  }_{L^2} . \nonumber
\end{align}
Next, in view of $\beta \geq 48$ chosen in the proof of Proposition \ref{proposition for th 2}, the small viscosity condition $\nu+\lambda \leq \frac{1}{2}$ and Lemma \ref{lemma for th 2}, we obtain
	\begin{align}
		w(t,k,\xi) \lesssim \nu ^{-\frac{1}{2}} , \nonumber
	\end{align}
and consequently
\begin{align}
	&\norm{\psi_{\pm}}_{L^2} + \frac{1}{M} \norm{\nabla \eta_{\pm}}_{L^2} + \norm{\omega_{\pm}}_{L^2}  \nonumber \\
	\lesssim & \nu^{-\frac{1}{2}}
	\xkh{
		\norm{ w^{-\frac{3}{4}} \widehat{\Psi}_{\pm} }_{L^2}
		+ \frac{1}{M} \norm{  w^{-\frac{3}{4}}\alpha^\frac{1}{2} \widehat{\Pi}_{\pm} }_{L^2}
		+ \norm{ w^{-\frac{3}{4}}(\widehat{F}_{\pm}-\nu M^2 \widehat{\Psi}_{\pm}  )  }_{L^2} 
	} \nonumber \\
	\lesssim &\: \nu^{-\frac{1}{2}}
	\mathrm{e}^{-\frac{\nu^\frac{1}{3}}{32}t}
	\xkh{ \frac{1}{M} \norm{\nabla \eta_{\pm}^{in}}_{L^2} +\norm{\psi_{\pm}^{in}}_{L^2} +\norm{ \eta_{\pm}^{in}+\omega_{\pm}^{in} -\nu M^2 \psi_{\pm}^{in} }_{L^2}   } \nonumber \\
	\lesssim &\: \nu^{-\frac{1}{2}}
	\mathrm{e}^{-\frac{\nu^\frac{1}{3}}{32}t}
	\xkh{
		\frac{1}{M} \norm{\nabla \eta_{\pm}^{in}}_{L^2} +\norm{\psi_{\pm}^{in}}_{L^2} 
		+ \norm{\omega_{\pm}^{in}}_{L^2}
	} ,
\end{align}
where in the third line we have applied Proposition \ref{proposition for th 2} and the last line follows from $\nu M^2 \leq 1$ and the periodicity on $x$ direction. To be more precise, 
	\begin{align}
		\label{poincare type inequality}
		&\norm{\eta^{in}_\pm }_{L^2(\mathbb{T} \times \mathbb{R})} = \norm{\widehat{\eta}^{in}_\pm(k,\xi)}_{L^2(\mathbb{Z} \times \mathbb{R})} \nonumber \\
		\leq& \norm{(k,\xi)\widehat{\eta}^{in}_\pm(k,\xi)}_{L^2(\mathbb{Z} \times \mathbb{R})}
		=\norm{\nabla \eta^{in}_\pm}_{L^2(\mathbb{T} \times \mathbb{R})},
	\end{align}
where the $k=0$ mode is zero, assumed at the beginning of this section. In some sence, it can be regarded as Poincar\'e's inequality with the domain being bounded on some direction.

We now proceed to prove the second result in Theorem \ref{th 2}. By Helmholtz decomposition and Plancherel Theorem, we have
	\begin{align}
		&\norm{u_{\pm}}_{L^2} + \frac{1}{M}\norm{\eta_{\pm}}_{L^2} \nonumber \\
		\leq& \norm{ \alpha^{-\frac{1}{2}}\widehat{\Psi}_{\pm} }_{L^2}
		+\norm{ \alpha^{-\frac{1}{2}}\widehat{\Gamma}_{\pm} }_{L^2}
		+\frac{1}{M}\norm{ \widehat{\Pi}_{\pm} }_{L^2} . \nonumber
	\end{align}
From Lemma \ref{lemma for th 2}, we deduce $w^\frac{3}{4} \alpha^{-\frac{1}{2}} \lesssim \nu^{-\frac{1}{6}}$ and consequently
	\begin{align}
		&\norm{u_{\pm}}_{L^2} + \frac{1}{M}\norm{\eta_{\pm}}_{L^2} \nonumber \\
		\lesssim &
		\: \nu^{-\frac{1}{6}} 
		\xkh{
			\norm{w^{-\frac{3}{4}} \widehat{\Psi}_{\pm} 		}_{L^2}
			+\norm{ w^{-\frac{3}{4}} \widehat{\Gamma}_{\pm} }_{L^2}
			+ \frac{1}{M} \norm{ w^{-\frac{3}{4}} \alpha^\frac{1}{2} \widehat{\Pi}_{\pm} }_{L^2}
		} \nonumber \\
		\lesssim &
		\: \nu^{-\frac{1}{6}} 	
		\begin{aligned}[t]
			\bigg(
			\norm{w^{-\frac{3}{4}} \widehat{\Psi}_{\pm} 		}_{L^2}
			&+\norm{w^{-\frac{3}{4}} (\widehat{F}_{\pm}-\nu M^2 		\widehat{\Psi}_{\pm}) }_{L^2}
			+M \frac{1}{M} 	\norm{w^{-\frac{3}{4}}\alpha^\frac{1}{2}\widehat{\Pi}_{\pm}  }_{L^2}\nonumber \\
			&+\nu M^2 \norm{w^{-\frac{3}{4}} \widehat{\Psi}_{\pm} }_{L^2}
			+\frac{1}{M}\norm{w^{-\frac{3}{4}}\alpha^\frac{1}{2}\widehat{\Pi}_{\pm}  }_{L^2} 
			\bigg)						
		\end{aligned}  \nonumber
	\end{align}
In view of $\nu M^2 \leq 1$ and Proposition \ref{proposition for th 2}, it holds that
\begin{align}
	&\norm{u_{\pm}}_{L^2} + \frac{1}{M}\norm{\eta_{\pm}}_{L^2} \nonumber \\
	\lesssim& 
	\: \nu^{-\frac{1}{6}} 	
	\xkh{ 
		\norm{w^{-\frac{3}{4}} \widehat{\Psi}_{\pm} }_{L^2}
		+\norm{w^{-\frac{3}{4}} (\widehat{F}_{\pm}-\nu M^2 \widehat{\Psi}_{\pm}) }_{L^2}
		+\frac{1}{M} 	\norm{w^{-\frac{3}{4}}\alpha^\frac{1}{2}\widehat{\Pi}_{\pm}  }_{L^2}
	} \nonumber \\
	\lesssim&
	\: \nu^{-\frac{1}{6}} 	
	\mathrm{e}^{-\frac{\nu^\frac{1}{3}}{32}t}
	\xkh{ \frac{1}{M} \norm{\nabla \eta_{\pm}^{in}}_{L^2} +\norm{\psi_{\pm}^{in}}_{L^2} +\norm{ \eta_{\pm}^{in}+\omega_{\pm}^{in} -\nu M^2 \psi_{\pm}^{in} }_{L^2}   }\nonumber \\
	\lesssim &
	\: \nu^{-\frac{1}{6}} 	
	\mathrm{e}^{-\frac{\nu^\frac{1}{3}}{32}t}
	\xkh{
		\frac{1}{M} \norm{\nabla \eta_{\pm}^{in}}_{L^2} +\norm{\psi_{\pm}^{in}}_{L^2} 
		+ \norm{\omega_{\pm}
			^{in}}_{L^2}
	} ,
\end{align} 
where in the last line we have also used (\ref{poincare type inequality}).
 Hence, the proof of Theorem \ref{th 2} is now complete.
\end{proof}
There is no loss of derivatives in Proposition \ref{proposition for th 2} and Theorem \ref{th 2}.

\end{document}